\documentclass[a4paper,11pt]{article}

\usepackage[scale=0.7,vmarginratio={1:2},heightrounded]{geometry}

\usepackage[utf8]{inputenc}
\usepackage{amsmath}
\usepackage{amssymb}
\usepackage{amsfonts}
\usepackage{amsthm}
\usepackage{enumitem}
\usepackage{xcolor}
\usepackage{authblk}
\usepackage{authblk}
\usepackage{float}
\usepackage{cleveref}
\usepackage{arydshln}
\usepackage{csquotes}
\usepackage{tikz}
\usetikzlibrary{positioning}
\usetikzlibrary{arrows}
\usetikzlibrary{decorations.pathreplacing}
\usepackage{blkarray}
\usepackage{caption}

\usepackage[pagewise]{lineno}

\DeclareMathOperator*{\arginf}{arg\,inf}
\newcommand{\R}{{\mathbb R}}
\newcommand{\C}{{\mathbb C}}
\newcommand{\Z}{\mathbb Z}
\newcommand{\Mcal}{\mathcal {M}}

\newcommand{\T}{\mathcal{T}}

\newcommand{\bfd}{\mathbf{d}}

\newcommand{\supp}{\mathrm{supp}}

\newcommand{\Fcal}{\mathcal{F}}
\newcommand{\F}{\mathcal{F}}
\newcommand{\tphiT}{\overline{\phi}_\T}
\newcommand{\tMT}{\widetilde{\Mcal}_\T}

\newcommand{\rowspan}{\mathrm{rowspan}}

\newtheorem{theorem}{Theorem}
\newtheorem{lemma}[theorem]{Lemma}
\newtheorem{prop}[theorem]{Proposition}
\newtheorem{cor}[theorem]{Corollary}

\theoremstyle{definition}
\newtheorem{definition}[theorem]{Definition}
\newtheorem{example}[theorem]{Example}
\newtheorem{remark}[theorem]{Remark}
\newtheorem{question}{Question}

\numberwithin{theorem}{section}

\newcommand{\ch}{\text{ch}}
\newcommand{\Scal}{\mathcal{S}}
\newcommand{\p}{p^*}
\newcommand{\q}{q^*}

\renewcommand{\alpha}{a}

\newcommand{\multipart}{multipartition }

\DeclareMathOperator*{\argmin}{arg\,min}
\DeclareMathOperator*{\argmax}{arg\,max}
\newcommand{\ds}{\displaystyle}

\title{\textbf{Rational partition models under iterative proportional scaling}}

\author[1,2]{Jane Ivy Coons}
\author[3]{Carlotta Langer}
\author[4]{Michael Ruddy}
\affil[1]{St John’s College, University of Oxford, United Kingdom}
\affil[2]{Mathematical Institute, University of Oxford, United Kingdom}
\affil[ ]{\textit {jane.coons@maths.ox.ac.uk}}
\affil[3]{Hamburg University of Technology, Hamburg, Germany}
\affil[ ]{\textit {carlotta.langer@tuhh.de}}
\affil[4]{Qventus}
\affil[ ]{\textit {mruddy@qventus.com}}

\date\today

\begin{document}
\maketitle

\begin{abstract}
In this work we investigate \emph{partition models}, the subset of log-linear models for which one can perform the iterative proportional scaling (IPS) algorithm to numerically compute the maximum likelihood estimate (MLE). Partition models include families of models such as hierarchical models and balanced, stratified staged trees. We define a sufficient condition, called the Generalized Running Intersection Property (GRIP), on the matrix representation of a partition model under which IPS algorithm produces the exact MLE in one cycle. Additionally we connect the GRIP to the toric fiber product and to previous results for hierarchical models and balanced, stratified staged trees. This leads to a characterization of balanced, stratified staged trees in terms of the GRIP.
\end{abstract}

\section{Introduction}

The iterative proportional scaling (IPS) algorithm is a simple and efficient numerical algorithm for computing the maximum likelihood estimate (MLE) for certain families of log-linear statistical models, which we call \emph{partition models}. The IPS algorithm is widely used in survey statistics, such as in \cite{Chang2021, Hauer2019, Yiannoutsose2013906118}, and has been researched in connection to optimal transport problems in its more restricted form as Sinkhorn's algorithm \cite{Berman2020, Pey2019}. Partition models include hierarchical models, which have been the subject of much study in connection with the IPS algorithm \cite{ET09, SH74, XSS16}. Maximum likelihood estimation for log-linear models and its relationship to algebraic and combinatorial objects is also of interest in algebraic statistics. For example in \cite{AKRS21}, the authors connect invariant theory and algorithms to compute the MLE including a form of the IPS algorithm. Many recent works have also explored the number of complex critical points of the log-likelihood function, also known as the ML-degree (see \cite{ABB19,CHKS06,CMR20,DMS19, Huh13, Huh14,STZ20}).

From an information-geometric perspective, calculating the MLE for a partition model can be described as projecting to linear families defined by the partitions of the matrix representing the model (see Definitions \ref{def:loglinear} and \ref{def:multipart}). Given a data vector and an estimate on the model, the IPS algorithm updates this estimate at each step by projecting onto a different linear family, converging towards the MLE. We say that the IPS algorithm has completed \emph{one cycle} after it has iterated through each linear family exactly once.

At each step of the IPS algorithm, the estimate is a rational function of the data vector. Thus in the case of one-cycle convergence, the MLE itself is a rational function of the data vector. Models for which the MLE can always be described as a rational function of the data vector are called \emph{rational} and are a subject of recent interest \cite{CS20, DMS19, HS14}. This is in contrast to what is sometimes called the \emph{generalized}\footnote{Many works will refer to the generalized IPS algorithm as simply ``the IPS algorithm'' while referring to its predecessor, the algorithm we consider in this work, as the ``classical IPS algorithm.''} IPS algorithm, another numerical algorithm for computing the MLE for log-linear models, which does not produce a rational function at each step, yet can be applied to any log-linear model.

In this work we are interested in the question, ``When does the IPS algorithm exactly produce the MLE after one cycle?'' We first note that the matrix representation of a partition model heavily influences the outcome of the IPS algorithm (see Example \ref{ex:DiffRepDiffIPS}). In \cite{SH74} the author defines the Running Intersection Property (RIP) for hierarchical models, which gives sufficient conditions on the matrix representation of a hierarchical model so that the IPS algorithm always produces the MLE exactly in one cycle. The author shows that for decomposable hierarchical models, there always exists such a representation. Drawing inspiration from the RIP, we define the Generalized Running Intersection Property (GRIP) on the matrix representations of general partition models and show that it gives sufficient conditions for the IPS algorithm to always produce the MLE exactly in one cycle. 
In the case of hierarchical models we show that the RIP is a special case of the GRIP (see Prop. \ref{prop:rip2grip}). We also connect previous work on the rationality of hierarchical models with the GRIP.

We additionally investigate another family of partition models: balanced and stratified staged trees. Staged tree models, also known as chain event graphs, are probability tree models that encode conditional independence relationships among events. They have been the subject of much study from an algebraic perspective in recent years \cite{AD19, DMS19, GMN21, SA08}.  In this work we connect the GRIP to balanced and stratified staged trees.
One can associate a tree to any matrix representation of a partition model, as described in Section \ref{ssec:StagedTrees}. 
We show that if a matrix representation satisfies the GRIP, then the associated tree must be a balanced and stratified staged tree. Moreover we show that the unique matrix with no repeated columns associated to a balanced and stratified staged tree satisfies the GRIP. 
In \cite{DS21}, the authors show that decomposable hierarchical models are also balanced staged tree models and claim that balanced staged trees are a natural generalization of such models. Our work also implies this result and supports this claim.

We make use of the toric geometry of a partition model in order to determine whether IPS achieves one-cycle convergence on a given parametrization of the model.
In particular, we show that when a parametrization of the model satisfies the GRIP, the model can be written as a \emph{toric fiber product} of smaller rational partition models \cite{tfp07}.
This in turn allows us to write the MLE of the larger partition model as a normalized product of the MLEs of the smaller models \cite{AKK20}, which facilitates our proof of one-cycle convergence.

Our work draws novel connections between the geometry of a rational partition model and the performance of the iterative proportional scaling algorithm applied to it. Moreover, it gives sufficient conditions for a partition model to be rational. If the partition model has a parametrization which satisfies the GRIP, then it is rational and one can read the MLE directly from the matrix of this parametrization (Corollary \ref{cor:GRIPMLE}). These results highlight the significant role that model representation plays in the performance of numerical computations of MLEs. We include a diagram of our main results in Figure \ref{Fig:Diagram}, in which the matrix $A$ is an integer matrix representing a log-linear statistical model.

\begin{center}
\scalebox{0.825}{
\begin{tikzpicture}[>=stealth, block/.append style={rectangle, draw=gray!80, minimum size=12pt}, rounded corners]
\draw[line width =0.5mm,  ->] (-4.5,0) -- node[above,midway] {\small Prop. \ref{prop:rip2grip}} (-2,0);
\draw[line width =0.5mm,  ->] (1.5,0) -- node[above,midway] {\small Thm. \ref{thm:GRIPOneCycle}} (4,0);
\draw[line width =0.5mm,  ->] (-0.5, 0) -- node[left,midway, yshift=5mm] {\small Thm. \ref{thm:IdealEqualsTFP}} (-0.5,1.75);
\draw[line width =0.5mm,  ->] ( 0.5, 2.5) -- node[right,midway,yshift=-2.5mm] {Prop. \ref{thm:TFP2GRIP}} (0.5, 0.75);
\draw[line width =0.5mm,  ->] (0.5, 0) -- node[right,midway,yshift=-2.4mm] {\small Thm. \ref{thm:GRIP2balanced}} (0.5,-1.75);
\draw[line width =0.5mm,  ->] (-0.5, -2.5) -- node[left,midway, yshift=5mm] {\small Thm. \ref{thm:balanced2GRIP}} (-0.5, -0.75);
\node[block,line width=0.5mm, fill = white, align=left] at (-6.5, 0) {$A$ represents a hierarchical \\ model satisfying the RIP \\ {\small \quad \textit{Definition \ref{def:RIP}} } };
\node[block,line width=0.5mm,fill = white, align=left] at (0,0) {$A$ satisfies the\textbf{ GRIP} \\ {\small \quad \textit{Definition \ref{def:GRIP}} } };
\node[block,line width=0.5mm,fill = white, align=left] at (6.5,0) {IPS converges \\ to the MLE in one cycle  };
\node[block,line width=0.5mm,fill = white, align=left] at (0,2.5) {The vanishing ideal $I(A)$  equals an \\ iterated toric fiber product of linear ideals };
\node[block,line width=0.5mm,fill = white, align=left] at (0,-2.5) {The tree associated with $A$ \\ is balanced and stratified };
\end{tikzpicture} }

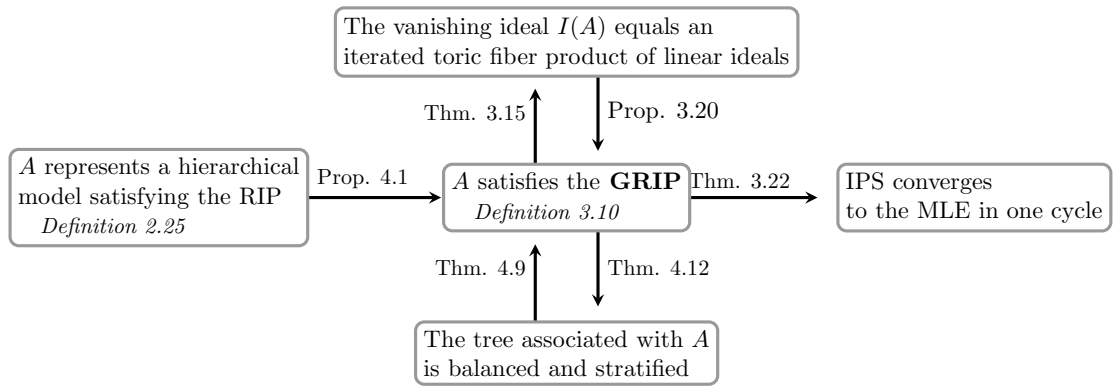
\captionof{figure}{Diagram of the main results of this work.} \label{Fig:Diagram}
\end{center}

%\vspace{0.25cm}
The work is structured as follows. In Section \ref{sec:prelim} we introduce some basic notions and previous results surrounding log-linear models, the MLE, the IPS algorithm, and some important partition model families. In particular, we define partition matrices, \multipart matrices, and partition models in \ref{ssec:loglinear}; describe the interative proportional scaling algorithm in \ref{ssec:IPS}; introduce the toric fiber product in \ref{ssec:TFP}; and discuss staged tree models and hierarchical models in \ref{ssec:StagedTrees} and \ref{ssec:hierarchical} respectively. In Section \ref{sec:GRIP} we introduce the GRIP and show this implies that the IPS algorithm produces the MLE after one cycle (Theorem \ref{thm:GRIPOneCycle}). In Section \ref{Sec:StagedTrees} we investigate the connection between the GRIP and staged tree models, showing that the GRIP characterizes a subset of such models. In Section \ref{Sec:HierarchicalModels} we show that the GRIP is indeed a generalization of the RIP. We end the work with a discussion of open problems stemming from our results in Section \ref{sec:disc}.

\section{Preliminaries}\label{sec:prelim}
\subsection{Log-Linear Models}\label{ssec:loglinear}
In this section we introduce the class of models to which we may apply the iterative proportional scaling algorithm in order to compute the maximum likelihood estimate. We start with some background on log-linear models before defining our main object of interest, \textit{partition models}.

Consider an $n\times m$ matrix $A=(a_{ij})$ such that $a_{ij} \in \Z_{\geq 0}$ and all column sums $\sum_{i=1}^d a_{ij}$ are equal. Then the matrix $A$ defines a homogeneous polynomial map $\phi_A:\R^n\rightarrow\R^m$ where
\begin{equation}\label{eq:AToricMap}
\phi_A(t_1,\hdots,t_n) = \left(\prod_{i=1}^n t_i^{a_{i1}},  \prod_{i=1}^n t_i^{a_{i2}},\hdots, \prod_{i=1}^n t_i^{a_{im}}\right).
\end{equation}
Let $\Delta_{m-1} \subset \R^m$ denote the open $(m-1)$-dimensional probability simplex. 
The {\bf support} of a row $\alpha$ of $A$, denoted $\mathrm{supp}(\alpha)$ is the set of column indices in which $\alpha$ has a nonzero entry. 

\begin{definition}\label{def:loglinear}
The \textbf{log-linear} model, denoted $\Mcal_A$, associated to the integer matrix $A$ described above is the intersection of the Zariski closure of the image of $\phi_A$ with the open probability simplex; that is,
$\Mcal_A = \overline{\text{Im}(\phi_A)} \cap \Delta_{m-1}.$
This is the discrete exponential family whose sufficient statistics are given by the rows of $A$. Hence it can be written as
\[
\Mcal_A = \left\lbrace p \in \Delta_{m-1} \vert \, \text{log} \, p \in \rowspan(A) \right\rbrace.
\]
\end{definition}

We note that several integer matrices with the vector of all ones in their rowspan may give rise to the same log-linear model.
In this work we are concerned with the relationship between the IPS algorithm and a particular choice of matrix $A$. This matrix should be thought of as a particular representation of the log-linear model $\Mcal_A$.

Note that the map $\phi_A$ is a monomial map. Hence the Zariski closure of $\Mcal_A$ is a \textit{toric variety}; for this reason the model $\Mcal_A$ is also referred to as a \textit{toric model} in the algebraic statistics literature. Let $I(A)$ denote the ideal of polynomials that vanishes on $\Mcal_A$. The following proposition, adapted from \cite[Prop~6.2.4]{sullivant2018} describes the relationship between the matrix $A$ and the ideal of polynomials vanishing on $\Mcal_A$.

\begin{prop}\label{Prop:SameRowspan}
Let $A \in \Z^{n \times m}$. 
 Then the vanishing ideal of $\Mcal_A$,
\begin{equation}\label{Eqn:ToricIdeal}
I(A) = \langle p^u - p^v \, |\, u,v \in \Z_{\geq 0}^m \,\,\text{and}\,\, Au=Av\rangle
\end{equation}
is the \textbf{toric ideal} of $A$ and $\Mcal_A$ is the intersection of the variety $V(I(A))$ with the open simplex $\Delta_{m-1}$.
Moreover, if $A' \in \Z^{n' \times m}$ satisfies $\rowspan(A) = \rowspan(A')$, then $I(A) = I(A')$ and $\Mcal_A = \Mcal_{A'}$.
\end{prop}

The integer matrices $A$ that we can apply IPS to must satisfy several conditions, which we outline below.

\begin{definition}\label{def:multipart}
A matrix $A \in \{0,1\}^{n \times m}$ is a {\bf \multipart matrix} if one can partition the rows of $A$ into submatrices $A^1,\dots,A^{k}$ such that in each $A^{\ell}$, $\ell \in \{1, \dots, k\}$, the entries of each column sums to 1. The matrices $A^1,\dots,A^{k}$ are called the {\bf partition matrices} of $A$. Such a matrix gives rise to a \textbf{partition model} $\Mcal_A$ and the monomial map that parametrizes this model as described in \eqref{eq:AToricMap} is a homogeneous and multilinear.
\end{definition}

\begin{remark}
In \cite{Rauh2011} the author uses the name \enquote{partition model} for the log-linear model associated to one partition matrix. He studies these models as an example of a family of log-linear models for which there exists a low constant that bounds the KL-divergence from an arbitrary point in the simplex to these models. Thereby he proves additional properties of the models defined by one partition matrix. 
\end{remark}

Let $A^1, A^2$ be partition matrices, each with $m$ columns.
Let $A^{1}A^{2}$ denote the matrix obtained by stacking $A^{1}$ above $A^{2}$; that is,
\begin{equation}\label{eq:BC}
A^{1}A^{2} := \begin{bmatrix}
A^{1} \\ \hline
A^{2}
\end{bmatrix}.
\end{equation}
From the partitions of $A$ we can build a new matrix, by stacking the partitions as defined above
\begin{equation}
A^{1, \dots, k} := A^{1}A^{2}\dots A^{k} \label{eq:PartitionMatrix}.
\end{equation}
There are many ways to arrange the rows of $A$ to build $A^{1, \dots, k}$; however, since they all have the same collection of rows as $A$ they clearly have the same rowspan, and hence define the same toric model. As we show in the next subsection, a different representation of the same model may affect the convergence of the iterative proportional scaling algorithm.  We assume throughout that $A$ is a \multipart matrix.

\begin{definition}
Let $\alpha^{\ell}_i$ denote the $i$-th row of the $\ell$-th partition, as illustrated in Example \ref{Ex:Partition}.
The  $I^{\ell}_i$ denote the set of column indices $j\in [m]$ such that the $j$-th entry of $\alpha^{\ell}_i$ is equal to 1. That is, $I^{\ell}_i$ is exactly the support of $a^\ell_i$. Let $n_{\ell}$ denote the number of rows of $A^{\ell}$.
\end{definition}

For a fixed partition $A^{\ell}$ and index $j\in \{1,\hdots,m\}$ there is exactly one row $\alpha^{\ell}_i$ with $j$ in its support. We  define the function $\Scal(\ell,j) \in \{1,\hdots, n_\ell\}$ where $\Scal(\ell,j)$ is the index such that $j\in I^\ell_{\Scal(\ell,j)}$. Thus $\alpha^\ell_{\Scal(\ell,j)}$ is the row of $A^\ell$ that has $j$ in its support.

\begin{example} \label{Ex:Partition} The matrix $A$ in Figure \ref{Fig:MatrixA} is an example of a \multipart matrix with three partitions. We have $m = 14, n_{1}, n_{2} = 2$ and $n_{3} = 4$. For instance, the support of the second row of the second partition is $I^{2}_{2} = \{4,5,6,7,11,12,13,14\}$ and the row of the first partition that has the tenth column in its support is $\Scal(1,10) = 2$.
\begin{figure}[H] \centering
\begin{tikzpicture}
\draw[] (0,-3.25) node{$A =
 \begin{pmatrix}
 1 & 1 & 1 & 1 & 1 & 1 & 1  &  \cdot & \cdot & \cdot & \cdot & \cdot & \cdot & \cdot\\
 \cdot & \cdot & \cdot & \cdot & \cdot & \cdot & \cdot  & 1 & 1 & 1 & 1 & 1 & 1 & 1 \\
 \hdashline
 1 & 1 & 1  & \cdot & \cdot & \cdot & \cdot & 1 & 1 & 1  & \cdot & \cdot & \cdot & \cdot\\
 \cdot & \cdot & \cdot  & 1 & 1 & 1 & 1 & \cdot & \cdot & \cdot & 1 & 1 & 1 & 1\\
 \hdashline
  1 & \cdot & \cdot & \cdot & \cdot & \cdot & \cdot & 1 & \cdot & \cdot & \cdot & \cdot & \cdot & \cdot \\
 \cdot & 1 & \cdot & \cdot & \cdot & \cdot & \cdot & \cdot & 1 & \cdot & \cdot & \cdot & \cdot & \cdot \\
 \cdot & \cdot & 1 & \cdot & \cdot & \cdot & \cdot & \cdot & \cdot & 1 & \cdot & \cdot & \cdot & \cdot \\
 \cdot & \cdot & \cdot & 1 & \cdot & 1 & \cdot & \cdot & \cdot & \cdot & 1 & \cdot & 1 & \cdot \\
 \cdot & \cdot & \cdot & \cdot & 1 & \cdot & 1 & \cdot & \cdot & \cdot & \cdot & 1 & \cdot & 1 \\
 \end{pmatrix} $};
\draw[] (0.1,-0.75) node{ \color{darkgray} $j \in \{$ \small 1,~ 2, $\cdots$  \hspace{4.1cm} $\cdots$ 13,~ 14 $\}$};
%\draw[] (0,-6) node{ $I^{2}_{2} = \{4,5,6,7,12,13,14,15\}$};
\draw[] (5.4,-1.25) node{ \color{darkgray} \footnotesize $\alpha^{1}_{1}$};
\draw[] (5.4,-1.75) node{ \color{darkgray} \footnotesize $\alpha^{1}_{2}$};
\draw[dashed, darkgray] (4.75,-2.04)--(6,-2.04);
\draw[dashed, darkgray] (4.75,-3.01)--(6,-3.01);
\draw[] (5.4,-2.25) node{ \color{darkgray} \footnotesize $\alpha^{2}_{1}$};
\draw[] (5.4,-2.75) node{ \color{darkgray} \footnotesize $\alpha^{2}_{2}$};
\draw[] (5.4,-3.25) node{ \color{darkgray} \footnotesize $\alpha^{3}_{1}$};
\draw[] (5.4,-3.75) node{ \color{darkgray} \footnotesize $\alpha^{3}_{2}$};
\draw[] (5.4,-4.25) node{ \color{darkgray} \footnotesize $\alpha^{3}_{3}$};
\draw[] (5.4,-4.75) node{ \color{darkgray} \footnotesize $\alpha^{3}_{4}$};
\draw[] (5.4,-5.2) node{ \color{darkgray} \footnotesize $\alpha^{3}_{5}$};
\draw [gray, decorate,-,decoration={brace,amplitude=5pt,mirror,raise=4pt},yshift=0pt]
(6,-2.03) -- (6,-1) node [darkgray,midway,xshift=0.8cm] {\footnotesize
 $A^{1}$ };
 \draw [gray, decorate,decoration={brace,amplitude=5pt,mirror,raise=4pt},yshift=0pt]
(6,-3) -- (6,-2.05) node [darkgray,midway,xshift=0.8cm] {\footnotesize
 $A^{2}$ };
 \draw [gray, decorate,decoration={brace,amplitude=5pt,mirror,raise=4pt},yshift=0pt]
(6,-5.4) -- (6,-3.02) node [darkgray,midway,xshift=0.8cm] {\footnotesize
 $A^{3}$ };
\end{tikzpicture} 
\end{figure}
 \captionsetup{width=.8\linewidth}
\captionof{figure}{A multipartition matrix with three partitions and the used row and column notation.} \label{Fig:MatrixA}
\end{example}

\begin{definition}\label{def:colweight}
For a matrix $A$ of the form \eqref{eq:PartitionMatrix}, the \textbf{column weight} of the $j$-th column is number of times the column is repeated and is denoted $c_j\in \Z_+$.
\end{definition}

For a \multipart matrix $A$, let $\bar{A}$ denote the $n\times \bar{m}$ matrix obtained by removing all repeated columns.   This is then a \textbf{$k$-way quasi-independence model}. In fact what we call partition models can also be thought of as $k$-way quasi-independence models with repeated columns. We refer the reader to \cite{CS20} for more information on the MLE of 2-way quasi-independence models.

We are able to apply the iterative proportional scaling algorithm as described in the next section to a partition model $\mathcal{M}_{A}$ in order to estimate the maximum likelihood of a data vector in $\mathbb{R}_{+}^{m}$. These are in fact the only integer matrices defining log-linear models for which we can apply the classical IPS algorithm.

\subsection{Iterative Proportional Scaling}\label{ssec:IPS}

The iterative proportional scaling (IPS) algorithm is a method to calculate the maximum likelihood estimator of a normalized data vector  with respect to the model $\mathcal{M}_{A}$. Before we discuss the algorithm in more detail, we first define the maximum likelihood method.

Broadly speaking, maximum likelihood estimation is a way to find a distribution in a statistical model, $\mathcal{M}_{A}$ that best fits some observed data. Let $u$ be the vector of counts of the observed data.
Assuming that the observations are independent and identically distributed (i.i.d.), the {\bf likelihood function} for $u$ with respect to the model is the function $L(p\mid u): \Mcal_A \rightarrow \R$ defined by
\[
L(p\mid u) = \prod_{i=1}^m p_j^{u_j}.
\] 
Evaluating the likelihood function at $p \in \Mcal_A$ yields the probability of observing the data $u$ from the distribution $p$. The {\bf maximum likelihood estimate} (or MLE) for $u$ is then
\[
\argmax_{p \in \mathcal{M}_{A}} L(p \mid u).
\]
That is, the MLE for $u$ is the distribution in the model that maximizes the probability of observing $u$.
We note that solving this optimization problem is equivalent to maximizing the {\bf log-likelihood function}, $\log(L(p\mid u))$, as the logarithm is concave. We use $d$ throughout to refer to \emph{normalized} data, and denote the MLE for  $d$ by $p^{\star}(d)$, or simply by $p^{\star}$ when $d$ is understood.  One can find an introduction to maximum likelihood estimation in most statistics resources, such as \cite{Harris1998, Klenke2007, Russel2011}. The MLE problem is also described from an algebraic perspective in \cite[Chapter~6]{sullivant2018}.

\begin{definition}
If $p^*(d)$ is a rational function in the entries of $d$ whenever it exists, then we say that $\Mcal_A$ has {\bf rational MLE}. If moreover $\Mcal_A$ is a partition model, then we call it a {\bf rational partition model}.
\end{definition}

The maximum likelihood estimate of any log-linear model can be characterized by the following result.
This result is sometimes referred to as Birch's theorem \cite{drton2008, Pachter2005}.

\begin{prop}[Corollary 7.3.9 in \cite{sullivant2018}] \label{prop:Birch}
Let $A$ be a matrix corresponding to a log-linear model and $d$ be the empirical distribution of the data. If the maximum likelihood estimate in the model $\mathcal{M}_{A}$ exists, then it is the unique solution to the equations
$A d = A p$ for $p \in \mathcal{M}_{A}.$
\end{prop}

The iterative proportional scaling, or iterative proportional fitting algorithm is a method for approximating the MLE in partition models. This is a well-known algorithm, first defined in the statistics literature by Deming and Stephan in 1940 in \cite{Deming1940} and analyzed further by for example Csiszár in \cite{csiszar1975} and Brown in \cite{Brown1993}. 
There are various types of iterative scaling alorithms. In \cite{drton2008} and \cite{sullivant2018} the authors describe a variant of the original algorithm, proposed as generalized iterative scaling by Darroch and Ratcliff  in \cite{darroch1972}.
We will focus on the earlier version defined below, because in this case each step is guaranteed to produce a rational function. 

In order to guarantee that the initial distribution lies in the partition model, we will fix the uniform distribution $p^{0} = (\frac{1}{n} , \dots , \frac{1}{n})$ as the starting point. The $\ell$th-step of the algorithm is then defined as 
\begin{equation}
p^{\ell} =  p^{\ell-1} \ast \frac{A^{i}d}{A^{i}p^{\ell-1}} \label{IPS}
\end{equation}
for $i = \ell$ mod $k$ where division and multiplication $\ast$ are defined componentwise. 
If $p^{r}$ is the MLE, then by Proposition \ref{prop:Birch}, $Ad = Ap^{r}$ and every factor $\frac{A^{i}d}{A^{i}p^{\ell-1}} = 1$ for $\ell > r$.
Thus if $p^r$ is the true MLE, then $p^{\ell} = p^r$ for all $\ell \geq r$.

Moreover, we note that every step is an information projection to the linear family $L^{i}$ defined by the $i$th partition $A^{i}$ as defined below. Indeed, the {\bf information projection} of $p$ to a non-empty, closed, convex set $M \subset \Delta_{m-1}$ is defined as the element $p' \in M$ that minimizes the KL-divergence between $M$ and $p$, that is,
\begin{equation*}
p' = \argmin\limits_{q \in M} D(q \parallel p) = \sum\limits_{j = 1}^{m} q_{j} \, log \, \frac{q_j}{p_j}.
\end{equation*}
The information projection of $p$ to the linear family  
\begin{equation}
 L^{i} = \{ p \in \Delta_{m-1} \vert \, A^{i}p = A^{i}d  \}\label{eq:InfProj}
\end{equation} 
is well-known and given by
$p' = p \ast \frac{A^{i}d}{A^{i}p}.$
This can be found, for instance, in \cite[Lemma~4.1]{csiszar2004}. Hence the algorithm defined in  (\ref{IPS}) performs an information projection to the linear family associated to a partition matrix $A^i$ at each step. This is sketched in Figure \ref{Fig:SketchIPS}.  
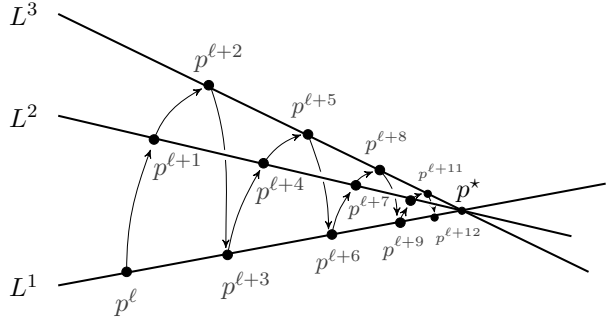
\begin{figure}[H]
\centering
\scalebox{0.9}{
\begin{tikzpicture}[rounded corners, >=stealth']
\draw[line width = 0.3mm] (0,0) -- (8,1.5);

\draw[] (-0.5,0) node{ $L^{1}$};
\draw[] (-0.5,2.5) node{ $L^{2}$};
\draw[] (-0.5,4) node{ $L^{3}$};

\draw[->] (2.2,3) to[out=-70,in=+90] (2.45,0.55);
\draw[fill = white, white] (2.46,1.925) circle (2pt);
\draw[->] (3.65,2.225) to[out=-70,in=+90] (3.95,0.825);
\draw[fill = white, white] (3.85,1.6) circle (2pt);
\draw[->] (4.7,1.7)  to[out=-50,in=+90] (4.95,1.025);
\draw[fill = white, white] (4.9,1.35) circle (2pt);
\draw[->] (5.4,1.35) to[out=-70,in=+100] (5.5,1.075);
\draw[fill = white, white] (5.475,1.2) circle (1.25pt);

\draw[line width = 0.3mm] (0,4) -- (7.75,0.2);
\draw[line width = 0.3mm] (0,2.5) -- (7.5,0.725);

\draw[fill = black] (1,0.2) circle (2pt);
\draw[darkgray] (1,-0.25) node{$p^{\ell}$};
\draw[fill = black] (1.4,2.15) circle (2pt);
\draw[darkgray] (1.75,1.75) node{$p^{\ell+1}$};
\draw[fill = black] (2.2,2.95) circle (2pt);
\draw[darkgray] (2.35,3.35) node{$p^{\ell+2}$};
\draw[fill = black] (2.475,0.45) circle (2pt);
\draw[darkgray] (2.75,0) node{$p^{\ell+3}$};
\draw[fill = black] (3,1.8) circle (2pt);
\draw[darkgray] (3.25,1.4) node{$p^{\ell+4}$};
\draw[fill = black] (3.65,2.225) circle (2pt);
\draw[darkgray] (3.75,2.65) node{\small $p^{\ell+5}$};
\draw[fill = black] (4,0.75) circle (2pt);
\draw[darkgray] (4.1,0.35) node{\small $p^{\ell+6}$};
\draw[fill = black] (4.35,1.475) circle (2pt);
\draw[darkgray] (4.55,1.15) node{\footnotesize $p^{\ell+7}$};
\draw[fill = black] (4.7,1.7) circle (2pt);
\draw[darkgray] (4.8,2.1) node{\footnotesize $p^{\ell+8}$};
\draw[fill = black] (5,0.925) circle (2pt);
\draw[darkgray] (5.05,0.55) node{ \footnotesize $p^{\ell+9}$};
\draw[fill = black, black] (5.15,1.25) circle (2pt);
\draw[fill = black] (5.4,1.35) circle (1.5pt);
\draw[darkgray] (5.6,1.67) node{ \tiny $p^{\ell+11}$};
\draw[fill = black,] (5.5,1) circle (1.5pt);
\draw[darkgray] (5.85,0.725) node{ \tiny $p^{\ell+12}$};
%\draw[fill = black] (5.6,1.17) circle (2pt);
%\draw[fill = black] (5.75,1.2) circle (2pt);
%\draw[fill = black] (5.8,1.05) circle (2pt);
\draw[->] (1,0.2) to[out=90,in=-110] (1.35,2);
\draw[->] (2.475,0.45) to[out=80,in=-120] (2.95,1.7);
\draw[->] (4,0.75) to[out=80,in=-125] (4.27,1.375);
\draw[->] (5,0.925) to[out=90,in=-120] (5.1,1.15);

\draw[->] (1.4,2.15) to[out=70,in=-150] (2.1,2.9);
\draw[->] (3,1.8) to[out=55,in=-160] (3.55,2.22);
\draw[->] (4.35,1.52) to[out=50,in=-160] (4.6,1.675);
\draw[->, line width = 0.15pt] (5.15,1.25) to[out=50,in=-160] (5.31,1.35);

\draw[fill = black] (5.9,1.1) circle (1.5pt);
\draw[] (6,1.4) node{$p^{\star}$};
\end{tikzpicture}}
 \captionsetup{width=.8\linewidth}
\caption{Sketch of the iterative proportional scaling algorithm in the case of three linear families.} \label{Fig:SketchIPS}
\end{figure}

The IPS algorithm converges to the unique point in the intersection between $\mathcal{L}_{A} = \bigcap_{i} L_{i}$ and $\overline{\mathcal{M}}_A$, as shown in Theorem 2.8 in \cite{ay2017}. Hence the result of the IPS algorithm minimizes the KL-Divergence with respect to the first argument
\begin{equation*}
p^{\star} = \arginf\limits_{p \in \mathcal{L}_{A}}  D(p \parallel q)  \text{ for all } q  \in \overline{\mathcal{M}}_A
\end{equation*}
as well as with respect to the second argument
\begin{equation*}
p^{\star} = \arginf\limits_{q \in \overline{\mathcal{M}}_{A}} D(p \parallel q)\text{ for all } p \in \mathcal{L}_{A}.
\end{equation*}
The last minimization is equivalent to maximizing the log-likelihood in case of a discrete model as discussed, for example, in \cite{drton2008}.

Note that since the uniform distribution is in $\overline{\mathcal{M}}_A$, minimizing with respect to the first argument results in the maximum entropy estimation. Hence it is related to the maximum entropy method proposed by Jaynes \cite{Jaynes}.

\begin{example}[The Independence model] \label{Ex:2x2Indep}
A \multipart matrix for the { 2x2} independence model is given below on the right where each dot represents a $0$. Let $d = (d_1, d_2, d_3, d_4)^T$ denote a vector of normalized data. Performing the first step of the algorithm results in:

\begin{minipage}[t]{0.65\textwidth}
{\small
\begin{align*}
p^{0}_{1} &= p^{0}_{2} = \frac{1}{4} \cdot \dfrac{d_{1} + d_{2}}{\frac{2}{4}} =  \frac{1}{2} (d_{1} + d_{2}), \\ p^{0}_{3} &= p^{0}_{4} = \frac{1}{4} \cdot \dfrac{d_{3} + d_{4}}{ \frac{2}{4}} =  \frac{1}{2} (d_{3} + d_{4})
\end{align*} }%
\end{minipage} \hfill
\begin{minipage}[t]{0.4\textwidth}
\vspace{-0.25cm}
\begin{align*}
A = \begin{pmatrix}
1 & 1  & \cdot & \cdot   \\
\cdot & \cdot & 1& 1 \\
\hdashline
1 & \cdot & 1 & \cdot \\
\cdot & 1 & \cdot & 1
\end{pmatrix}
\end{align*} 
\end{minipage}

In this case the algorithm converges to the MLE with the second step: 
{\small
\begin{align*}
p^{1}_{1} &= \frac{1}{2} (d_{1} + d_{2})  \cdot \dfrac{d_{1} + d_{3}}{ \frac{1}{2} (d_{1} + d_{2}) +  \frac{1}{2} (d_{3} + d_{4}  )} = (d_{1} + d_{2})(d_{1} +d_{3})\\
p^{1}_{2} &=\frac{1}{2} (d_{1} + d_{2})  \cdot \dfrac{d_{2} + d_{4}}{ \frac{1}{2} (d_{1} + d_{2}) +  \frac{1}{2} (d_{3} + d_{4})  } = (d_{1} + d_{2})(d_{2} +d_{4})\\
p^{1}_{3} &=\frac{1}{2} (d_{3} + d_{4}) \cdot \dfrac{d_{1} + d_{3}}{\frac{1}{2} (d_{1} + d_{2}) +  \frac{1}{2} (d_{3} + d_{4})} = (d_{3} + d_{4})(d_{1} + d_{3})\\
p^{1}_{4} &=\frac{1}{2} (d_{3} + d_{4}) \cdot \dfrac{d_{2} + d_{4}}{\frac{1}{2} (d_{1} + d_{2}) +  \frac{1}{2} (d_{3} + d_{4} )} = (d_{3} + d_{4})(d_{2} + d_{4})\\
\end{align*} }%
\end{example}
Note that at each step of the algorithm the approximation of the MLE is a rational function of the data $d$. Thus if after finitely many steps the algorithm results in the MLE of $d$ with respect to $\Mcal_A$, the model has a rational maximum likelihood estimator. The question naturally arises: for a partition model with rational MLE, does iterative proportional scaling always result in the MLE after finitely many steps?

\begin{example}\label{ex:DiffRepDiffIPS}
Here we consider the two matrices $A$ and $\tilde{A}$ depicted on the right below.

\noindent
\begin{minipage}[t]{0.7\textwidth}
%\vspace{-0.5cm}
Both matrices have rowspan equal to $\R^3$, so by Proposition \ref{Prop:SameRowspan}, $\mathcal{M}_{A} = \mathcal{M}_{\tilde{A}}$. Although they represent the same model, the convergence of the IPS algorithm is heavily influenced by the chosen representation. Let $d = (d_1, d_2, d_3)^T$ be normalized data. Using matrix $\tilde{A}$ the IPS algorithm converges in one step to the MLE, $p^{\star} = d$.
\linebreak

However, the  IPS algorithm on $A$ with normalized data vector $d$ results in
\end{minipage} \hfill
\begin{minipage}[t]{0.25\textwidth}
\vspace{-0.75cm}
\begin{align*}
A &= \begin{pmatrix}
1 & 1  & \cdot    \\
\cdot & \cdot & 1\\
\hdashline
1 & \cdot & \cdot \\
\cdot & 1 & 1
\end{pmatrix} \\
\tilde{A} &= \begin{pmatrix}
1 & \cdot  & \cdot    \\
\cdot & 1 & \cdot \\
\cdot & \cdot & 1\\
\end{pmatrix}
\end{align*}
\end{minipage}

{ \small
\begin{align*}
p^{0} &= \left(  \frac{1}{2} (d_{1} + d_{2}), \, \frac{1}{2} (d_{1} + d_{2}) , d_{3}   \right)
\\
p^{1} &= \left(  d_{1} ,\,  \frac{1}{2} (d_{1} + d_{2}) \dfrac{(d_{2} + d_{3})}{ \frac{1}{2} (d_{1} + d_{2}) + d_{3}}, d_{3} \dfrac{(d_{2} + d_{3})}{\frac{1}{2} (d_{1} + d_{2}) + d_{3}}  \right)
\\
p^{2} &=  \left(  d_{1} \dfrac{(d_{1} + d_{2})}{ d_{1} +  \dfrac{\frac{1}{2} (d_{1} + d_{2})(d_{2} + d_{3})}{ \frac{1}{2} (d_{1} + d_{2}) + d_{3}} } , \, \frac{1}{2} (d_{1} + d_{2}) \dfrac{(d_{2} + d_{3})}{ \frac{1}{2} (d_{1} + d_{2}) + d_{3}} \dfrac{(d_{1} +d_{2})}{ d_{1} + \dfrac{\frac{1}{2} (d_{1} + d_{2}) (d_{2} + d_{3})}{ \frac{1}{2} (d_{1} + d_{2}) + d_{3}} } , d_{3}  \right)
\end{align*} }%
In general the different projections have the following form:
\begin{alignat}{4} p^{k} =& \Biggl( d_{1}, \,  && p^{k-1}_{2} \dfrac{(d_{2} + d_{3})}{p^{k-1}_{2} + d_{3}}, \, &&d_{3} \dfrac{(d_{2} + d_{3})}{p^{k-1}_{2} + d_{3}} &&\Biggr), \quad k \, \text{odd} \label{IPSkodd} \\
p^{k} =& \Biggl( d_{1} \dfrac{(d_{1} + d_{2})}{d_{1} + p^{k-1}_{2}} , \,  && p^{k-1}_{2} \dfrac{(d_{1} + d_{2})}{d_{1} + p^{k-1}_{2} }, &&d_{3}   &&\Biggr), \quad k \, \text{even} \label{IPSkeven}
\end{alignat}
Suppose that there exists an index $k$ such that the second entry of $p^{k}$ is exactly $d_{2}$ in \eqref{IPSkodd} as well as \eqref{IPSkeven}. Then we can deduce that $d_{2} = p_{2}^{k-1}.$ 
Hence IPS can only result in the exact result if $d_{1} = d_{2}$.  Note that while the model is symmetric, the application of the IPS is not. Hence if we would reorder the partitions, then exact convergence would require $d_2 = d_3$.

In a practical evaluation with $20 \, 000$ random input distributions, the arithmetic mean of the iteration steps taken to get a step size smaller than $10^{-8}$ was approx $113,47$ with a minimum value of 8 and a maximum value of $287 \, 478$. Recall that in case of $\tilde{A}$ the necessary number of iteration steps is exactly 1. 
\end{example}

\subsection{Toric Fiber Products}\label{ssec:TFP}

In this section, we review an algebraic construction, known as the \emph{toric fiber product}, which will facilitate the proofs in the following sections. 
Toric fiber products were first introduced by Sullivant, and this exposition follows the notation of \cite{tfp07}. 
For the purposes of this article, we may broadly think of the toric fiber product as describing a way
to concatonate the parametrizations of two log-linear models that preserves some of their properties. 

For any positive integer $r$, let $[r] := \{ 1,\dots,r\}$. Fix positive integers $r$ and $s_i, t_i$ for each $i \in [r]$. 
Define three polynomial rings,
\begin{align*}
\C[x] & =  \C[x^i_j \mid i \in [r], j \in [s_i]], \\
\C[y] & =  \C[y^i_k \mid i \in [r], k \in [t_i]], \text{ and } \\
\C[z] & =  \C[z^i_{jk} \mid i \in [r], j \in [s_i], k \in [t_j]].
\end{align*}

Recall that a {\bf multigrading} on a polynomial ring $\C[t] := \C[t_1,\dots t_d]$ is an assignment of a multidegree vector to each monomial, $\deg(t^u) \in \Z^s$ that satisfies $\deg(t^{u+v}) = \deg(t^u) \deg(t^v)$. An ideal $I \in \C[t]$ is homogeneous with respect to this multigrading if each of its terms has the same multidegree.

Fix multigradings on $\C[x]$ and $\C[y]$ given by
$\deg(x^i_j) = \deg(y^i_j) = \bfd^i$
for integer vectors $\bfd^i$. Note that the multigrading of each indeterminate depends only on its superscript, $i$. 
Let $D$ be the matrix with columns $\bfd^i$ for $i \in [r]$.

Let $I \subset \C[x]$ and $J \subset \C[y]$ be homogeneous ideals with respect to this multigrading.
We define a ring homomorphism,
\begin{alignat*}{4}
\psi_{I,J}: & \quad & \C[z] & \rightarrow & (\C[x] / I) \otimes_{\C} (\C[y] / J) \\
& & z^i_{jk} & \mapsto & x^i_j \otimes_{\C} y^i_k
\end{alignat*}

\begin{definition}
The \textbf{toric fiber product} of $I$ and $J$ with multigrading $D$ is the ideal 
$I \times_D J = \mathrm{ker}(\psi_{I,J})$ in $\C[z]$.
\end{definition}

We assume throughout that $I$ and $J$ are toric, as the ideals associated to rational partition models are always toric.
In order to parametrize the toric fiber product of $I$ and $J$, we may simply take the product of their parametrizations.
More precisely, if $\phi_I$ and $\phi_J$ are monomial parametrizations of $I$ and $J$, then we obtain a parametrization for their toric fiber product according to multigrading $D$ via $z^i_{jk} \mapsto \phi_I(x^i_j)\phi_J(y^i_k).$

Now assume further that $D$ is linearly independent. In this case, we can explicitly describe a generating set for $I \times_D J$ using the following families of polynomials. First, let $i \in [r]$. Then we define
\[
\mathrm{Quad}^i := \{ z^i_{jk}z^i_{j'k'} - z^i_{jk'}z^i_{j'k} \mid j,j' \in [s_i], k,k' \in [t_i] \}.
\]
By replacing each indeterminate with its parametrization in terms of $\phi_I$ and $\phi_J$, we may see that each element of $\mathrm{Quad}^i$ belongs to $I \times_D J$. We define
$\mathrm{Quad} := \cup_{i=1}^r \mathrm{Quad}^i$.

Let $f = \prod_{\alpha=1}^d x^{i_{\alpha}}_{j_{\alpha}} - \prod_{\alpha=1}^d x^{i_{\alpha}}_{j'_{\alpha}} \in I$. Note that all binomials in $I$ can be written in this form as $D$ is linearly independent and $I$ is homogeneous with respect to the multigrading specified by $D$. We define the set of binomials,
\[
\mathrm{Lift}(f) = \left\lbrace  \prod_{\alpha=1}^d z^{i_{\alpha}}_{j_{\alpha}k_{\alpha}} - \prod_{\alpha=1}^d z^{i_{\alpha}}_{j'_{\alpha} k_{\alpha}} 
\mid k_{\alpha} \in [t_{i_{\alpha}}] \text{ for all } \alpha \in [d]\right\rbrace.
\]
Observe that each binomial in $\mathrm{Lift}(f)$ also lies in $I \times_D J$. For each binomial $g \in J$, we define $\mathrm{Lift}(g)$ analogously. Let $F$ be a generating set for $I$. Then we define the family of polynomials,
$\mathrm{Lift}(F) := \cup_{f \in F} \mathrm{Lift}(f),$
and similarly for any generating set of $J$. We are now able to describe the generating set, and in fact, a Gr\"obner basis, for $I \times_D J$. The following is an adaptation of Theorem 2.9 of \cite{tfp07} for our purposes.

\begin{theorem}[\cite{tfp07}] \label{thm:TFPgenerators}
Let $I$ and $J$ be toric ideals that are homogeneous with respect to the multigrading specified by $D$. Suppose further that the columns of $D$ are linearly independent. Let $F$ be a Gr\"obner basis for $I$  and let $G$ be a Gr\"obner basis for $J$. Then
\[
\mathrm{Quad} \cup \mathrm{Lift}(F) \cup \mathrm{Lift}(G)
\]
is a Gr\"obner basis for $I \times_D J$ with respect to a certain weight order.
\end{theorem}

\begin{example}
Consider the matrices
\[
B = \begin{blockarray}{[cc|cc]}
1 & 1 & 0 & 0 \\
0 & 0 & 1 & 0 \\
0 & 0 & 0 & 1
\end{blockarray} \qquad \text{ and } \qquad C = \begin{blockarray}{[cc|cc]}
1 & 1 & 0 & 0 \\
0 & 0 & 1 & 1 \\
1 & 0 & 1 & 0 \\
0 & 1 & 0 & 1
\end{blockarray}
\]
The columns of $B$ correspond in order to variables $x^1_1, x^1_2, x^2_1, x^2_2$ and the columns of $C$ correspond in order to variables $y^1_1, y^1_2, y^2_1, y^2_2$. Let $I$ be the vanishing ideal of the log-linear model specified by $B$ and let $J$ be that of $C$. We have
$I = \langle x^1_1 - x^1_2 \rangle$ and $\qquad J = \langle y^1_1y^2_2 - y^1_2 y^2_1\rangle.$
Consider the multigrading with respect to the $2 \times 2$ identity matrix $D$, where $\deg(x^1_j) = \deg(y^1_j) = e_1$ and $\deg(x^2_j) = \deg(y^2_j) = e_2$ for $j = 1,2$. Observe that both $I$ and $J$ are homogeneous with respect to this multigrading. The matrix defining the parametrization of their toric fiber product, $I \times_D J$, is
\[
\begin{blockarray}{cccc|cccc}
z^1_{11} & z^1_{12} & z^1_{21} & z^1_{22} & z^2_{11} & z^1_{12} & z^1_{21} & z^1_{22} \\
\begin{block}{[cccc|cccc]}
 1 & 1 & 1 & 1 & 0 & 0 & 0 & 0 \\
0 & 0 & 0 & 0 & 1 & 1 & 0 & 0 \\
 0 & 0 & 0 & 0 & 0 & 0 & 1 & 1 \\
\BAhline
 1 & 1 & 1 & 1 & 0 & 0 & 0 & 0 \\
 0 & 0 & 0 & 0 & 1 & 1 & 1 & 1 \\
 1 & 0 & 1 & 0 & 1 & 0 & 1 & 0 \\
 0 & 1 & 0 & 1 & 0 & 1 & 0 & 1 \\
\end{block}
\end{blockarray}
\]

In order to construct a generating set for $I \times_D J$, we consider the lifts of the binomials that generate $I$ and $J$. We have
\begin{align*}
\mathrm{Lift}(x^1_1 - x^1_2) &=&& \{z^1_{11} - z^1_{21}, z^1_{12} - z^1_{22}\} \\
\mathrm{Lift}(y^1_1y^2_2 - y^1_2 y^2_1) &=&& \{ z^1_{11}z^2_{12} - z^1_{12} z^2_{11}, z^1_{11}z^2_{22} - z^1_{12} z^2_{21}, z^1_{21}z^2_{12} - z^1_{22} z^2_{11}, z^1_{21}z^2_{22} - z^1_{22} z^2_{21} \}
\end{align*}
For each $i = 1,2$, we have
$\mathrm{Quad}^i = \{z^i_{11}z^i_{22} - z^i_{12} z^i_{21} \}.$
By Theorem \ref{thm:TFPgenerators}, we have that
\[
\mathrm{Lift}(x^1_1 - x^1_2) \cup \mathrm{Lift}(y^1_1y^2_2 - y^1_2 y^2_1)  \cup \mathrm{Quad}^1 \cup \mathrm{Quad}^2
\]
is a generating set for $I \times_D J$.
\end{example}

We conclude this introduction to the toric fiber product by describing the MLE of the model associated to the toric fiber product in terms of the MLE of its factors. Let $B, C$ be matrices representing the log-linear models associated to $I$ and $J$. Let $\hat{p}(B), \hat{p}(C)$ and $\hat{p}(D)$ denote the maximum likelihood estimators for the log-linear models given by $B, C$ and $D$, respectively. The following is Theorem 5.5 of \cite{AKK20} and plays a critical role in our proofs of the results in Section 3.

\begin{theorem}[\cite{AKK20}]\label{thm:AKK20}
If $D$ has linearly independent columns, then the $(i,j,k)$th coordinate function of the maximum likelihood estimator for the log-linear model associated to the toric fiber product $I \times_D J$ is
\[
\hat{p}^i_{jk} = \frac{\hat{p}^i_j(B) \hat{p}^i_k(C)}{\hat{p}^i(D)}.
\]
\end{theorem}

\subsection{Balanced Staged Tree Models}\label{ssec:StagedTrees}

We introduce staged tree models following the presentation in \cite{AD19}. Staged tree models are a type of probability tree model, which are graphical models that encode conditional independence statements between events. This is in contrast to hierarchical models, which instead encode conditional independence of random variables. For a detailed introduction see \cite{CGS18}.  
Every discrete Bayesian network and decomposable graphical model is a staged tree model \cite{DS21}.
In this section, we define the notion of balanced and stratified staged trees, and we describe the correspondence between balanced, stratified staged tree models and partition models with a rational maximum likelihood estimator.

Let $\mathcal{T} = (V,E)$ be a tree with vertex set $V$ and edge set $E$. A directed tree is \textit{rooted} if there exists a vertex $v_0$, which we called the \textit{root}, such that the unique path from $v_0$ to each leaf in $T$ is a directed path.

For a particular vertex $v\in V$, we can define the set of \textit{children} of $v$ as the set $\ch(v) = \{ u\, |\, (v,u)\in E\}$. Then the outgoing edges from $v$ can be denoted as the set $E(v) = \{(v,u)\, |\, u\in \ch(v)\}$, and $v$ is a leaf if $E(v) = \emptyset$. A tree is \textit{labeled} if there exists a label set $S$ and a surjective mapping, called a labelling, $\theta : E\rightarrow S$. From now on we assume that $(\T, \theta)$ is a rooted labeled tree graph.

\begin{definition}
For any vertex $v\in V$, the set of labels of its outgoing edges is the \textbf{floret} of $v$, denoted $\F_v := \{ \theta(e)\,|\, e\in E(v)\}$. The set of florets of $(\T, \theta)$ is denoted $F_\T$.
\end{definition}

\begin{definition}
The labeled tree graph $(\T, \theta)$ is a \textbf{staged tree} if the following are satisfied:
\begin{itemize}
\item For each $v\in V$, no two edges in $E(v)$ have the same labelling, i.e. $|\theta_v| = |E(v)|$.

\item For any two vertices $v,w\in V$, the florets $\F_v$ and $\F_w$ are equal or disjoint. When $\F_v=\F_w$ we say that the vertices $v$ and $w$ are in the same \textit{stage}.
\end{itemize}
\end{definition}

A \textit{staged tree model} can be associated to a staged tree $(\T,\theta)$ as follows. Let $\Lambda$ denote the set of root-to-leaf paths in $\T$ and suppose $|\Lambda|=m$. For $\lambda\in \Lambda$ let $E(\lambda)$ be the set of edges in $\Lambda$ and define the multiset $\theta(\lambda):=\{\{s_i\, |\, \theta(e)=s_i\,\, \text{for}\,\, e\in E(\lambda) \}\}$. The paths $\lambda$ parameterize the model $\Mcal_\T$ defined below.

\begin{definition}\label{def:StagedTree}
For a staged tree $(\T,\theta)$ with label set $S=\{s_1,\hdots,s_r\}$, the \textbf{staged tree model} $\Mcal_\T$ is the image of the map $\phi_{\T} : \Theta \rightarrow \Delta_{m-1}$ defined by
\begin{equation}\label{eq:TreeToricMap}
\phi_{\T}(s_1,\hdots, s_{r}) = \left(\prod_{s_i \in \theta(\lambda)} s_i \right)_{\lambda\in\Lambda}
\end{equation}
where the parameter space is 
\begin{align*}
\Theta := \left\lbrace (s_1,\hdots,s_r) \in (0,1]^r \, |\, \sum_{s_i\in \mathcal{F}} s_i =1\,\, \text{for all florets}\,\, \mathcal{F}\in F_\T \right\rbrace.
\end{align*}
\end{definition}

\newpage
\begin{example}\label{ex:firststagedtree} Consider the staged tree $(\T,\theta)$ in Figure \ref{Fig:Tree}.

\noindent
\begin{minipage}[t]{0.45\textwidth}
%\vspace{-0.5cm}
In this diagram, the vertices in the same stage have the same coloring. We can immediately see that each floret has equal or disjoint label set. We can refer to the vertices of the tree by the edge labels defining the path to the vertex, i.e. the blue vertices are $v_{s_0t_0}$ and $v_{s_1t_0}$. If $v_0$ is the root vertex then the florets of the tree are
\begin{align*}
\F_{v_0}, \quad\textcolor{violet}{\F_{v_{s_0}}} &= \textcolor{violet}{\F_{v_{s_1}}}, \\ \textcolor{blue}{\F_{v_{s_0t_0}}} &=  \textcolor{blue}{\F_{v_{s_1t_0}}},\quad \\ \textcolor{red}{\F_{v_{s_0t_1}}} &= \textcolor{red}{\F_{v_{s_1t_1}}}.
\end{align*}
\end{minipage} \hfill
\begin{minipage}[t]{0.5\textwidth}
\begin{center} 
\centering\raisebox{\dimexpr \topskip-\height}{% 

\scalebox{0.8}{ 
\begin{tikzpicture}[grow=right, edge from parent/.style={draw,-latex}, rounded corners]
\node[circle, draw, line width = 0.75pt]  {}
child[level distance = 16mm, sibling distance=3cm, line width = 0.75pt] {node[circle,draw, fill= violet!40] {} { 
	child[sibling distance=1.5cm, level distance = 18mm]{node[circle,draw,fill= red!50] {}
		child[sibling distance=0.5cm]{node[circle,draw, red!50] {}
		edge from parent node [below, align=center] {}}
		child[sibling distance=0.5cm]{node[circle,draw, red!50] {}
		edge from parent}
	 edge from parent node [->, below, align=center]{}          
	}
	child[sibling distance=1.5cm , level distance = 18mm]{node[circle,draw, fill=blue!30] {}
		child[sibling distance=0.5cm]{node[circle,draw, blue!50!white] {}}
		child[sibling distance=0.5cm]{node[circle,draw, blue!50!white] {}}
		child[sibling distance=0.5cm]{node[circle,draw, blue!50!white] {}}     
	}                                
}edge from parent node [left, align=center] {\small $s_{1}$}} 
child[level distance = 16mm,sibling distance=3cm, line width = 0.75pt] {node[circle,draw, fill= violet!40] {} 
	child[sibling distance=1.5cm, level distance = 18mm]{node[circle,draw, fill=red!50] {} 
		child[sibling distance=0.5cm]{node[circle,draw, red!50] {}
		edge from parentnode [below, align=center, xshift=5pt] {\small $r_{4}$}}
		child[sibling distance=0.5cm]{node[circle,draw, red!50] {}
		edge from parent node [above, align=center, xshift=5pt] {\small $r_{3}$}}   
		 edge from parent node [->, above, align=center]
                {\small $t_{1}$}               
	}
	child[sibling distance=1.5cm, level distance = 18mm]{node[circle,draw, fill=blue!30]{}
		child[sibling distance=0.5cm,draw]{node[circle,draw, blue!50!white] {}
		edge from parent node [below, align=center, xshift=5pt] {\small $r_{2}$}}
		child[sibling distance=0.5cm]{node[circle,draw, blue!50!white] {}
		edge from parent node [above=-3pt, align=center, xshift=5pt] {\small $r_{1}$}}
		child[sibling distance=0.5cm]{node[circle,draw, blue!50!white] {}
		edge from parent node [above, align=center, xshift=5pt] {\small $r_{0}$}}
		 edge from parent node [->, above, align=center]
                { \small $t_{0}$}     
	}   edge from parent node [left, align=center] {\small $s_{0}$}                        
};
\draw[] (6,2.75) node{\footnotesize $p_{000}$};
\draw[] (6,2.25) node{\footnotesize $p_{001}$};
\draw[] (6,1.75) node{\footnotesize $p_{002}$};
\draw[] (6,1) node{\footnotesize $p_{010}$};
\draw[] (6,0.5) node{\footnotesize $p_{011}$};
\draw[] (6,-0.25) node{\footnotesize $p_{100}$};
\draw[] (6,-0.75) node{\footnotesize $p_{101}$};
\draw[] (6,-1.25) node{\footnotesize $p_{102}$};
\draw[] (6,-2) node{\footnotesize $p_{110}$};
\draw[] (6,-2.5) node{\footnotesize $p_{111}$};
\end{tikzpicture}  }  }
 \captionsetup{width=.8\linewidth} 
\captionof{figure}{Staged tree associated to the model $\Mcal_\T$.} \label{Fig:Tree}
\end{center}
\end{minipage}
\vspace{0.25cm}

The associated staged tree model $\Mcal_\T$ is image of the map $\phi_\T: \Theta\rightarrow \Delta_{9}$ where
\begin{align*}
\phi_\T(s_0,s_1,t_0,t_1,r_0,r_1,r_2) = (s_0t_0r_0, &s_0t_0r_1, s_0t_0r_2, s_0t_1r_3,s_0t_1r_4,\\
 &s_1t_0r_0, s_1t_0r_1, s_1t_0r_2, s_1t_1r_3,s_1t_1r_4 ),
\end{align*}
and
\begin{align*}
\Theta = \{ (s_0,s_1,t_0,t_1,r_0,r_1,r_2) \in (0,1]^7\, |\, s_0+s_1=1, &t_0+t_1=1, \\
&r_0+r_1+r_2=1, r_3+r_3+r_4=1\}.
\end{align*} 

\end{example}

\begin{remark}
We use the slightly more general definitions of staged tree and staged tree models introduced in \cite{AD19} that allow for staged trees with singleton florets. The authors show that the staged tree obtained by contracting these florets, which satisfies the definition used in previous literature (i.e. \cite{CGS18, GS18}), results in the same staged tree model \cite[Lemma 5.11]{AD19}.
\end{remark}

The map $\phi_\T$ in Definition \ref{def:StagedTree} also defines a toric model if one ``forgets'' the conditions imposed on the parameter space by the florets. We denote the map obtained by extending $\phi_\T$ to all of $\R^r$ as $\tphiT: \R^r \rightarrow \R^m$ and the corresponding toric model as $\tMT$. Immediately it follows that $\Mcal_\T \subset \tMT$, but in general it is not true that $\Mcal_\T = \tMT$. In \cite{DG20} the authors found necessary and sufficient graphical conditions on $(\T,\theta)$ such that $\Mcal_\T = \tMT$. We detail those conditions in Section \ref{Sec:StagedTrees}.

\begin{definition}
For any vertex in a staged tree $(\T,\theta)$, we define the \textbf{level} $\ell(v)$ as the number of edges in the unique root-to-$v$ path. A staged tree $(\T,\theta)$ is \textbf{stratified} if all its leaves have the same level and for any two vertices $v,w$ such that $\mathcal{F}_v=\mathcal{F}_w$, we have $\ell(v)=\ell(w)$. We say that the \textbf{level} of a stratified staged tree is the level of its leaves.
\end{definition}

There is a natural partition of the label set of a stratified staged tree by level. For a stratified staged tree of level $\ell$, we can write $S = \{s^0_0,s^0_1,\hdots,s^0_{n_0},s^1_0,\hdots,s^{\ell}_{n_{\ell}}\}$ where $n_k$ is the number of distinct edge labels for edges emanating from vertices with level $k$. Thus the monomial map  $\phi_\T$ for a stratified staged tree of level $\ell$ is multilinear and homogeneous, which implies that the associated matrix is a \multipart matrix with $\ell$ partitions. We denote the matrix associated with $\phi_\T$ as $A_\T$.

Similarly with \multipart matrix $A$, we can associate a directed, rooted tree with a labelling induced by the rows of $A$. We define the label set associated with $A$ as $S_A=\{s^1_1,\hdots,s^1_{n_1},\hdots,s^{\ell}_1,\hdots,s^{\ell}_{n_{\ell}}\}$ associated to the row vectors $\alpha^1_1, \alpha^1_2,\hdots, \alpha^{\ell}_{n_{\ell}}$.

Starting with a root vertex, we add a labeled edge $s^1_i$ for each associated row vector in $A^1$. To each edge $s^1_i$, we attach an edge labeled $s^2_j$ to the outgoing vertex for each $s^2_j$ is the label set $S_A$ such that $I^1_i\cap I^2_j\neq \emptyset$. The resulting directed, rooted, and labeled tree graph is $\T_{A^{1,2}}$ where $A^{1,2}$ is the matrix with partitions $A_1, A_2$. The paths of $\T_{A^{1,2}}$ exactly correspond to the distinct columns of $A^{1,2}$. We note that if $A^{1,2}$ has two columns that are equal to one another, these correspond to the same path in $\T_{A^{1,2}}$. We repeat this process for each $A^k$. In this way we can define $\T_A$ as the directed, rooted tree graph with labelling such that the paths in $\T_A^{(\ell)}$ correspond to the distinct columns of $A^{1,\hdots,\ell}$.

We will see that staged trees are useful in describing certain properties of \multipart matrices. In particular, when we introduce conditions under which \multipart matrices produce the MLE after one IPS cycle, we require that $\T_A$ is a staged tree. We also show in Section \ref{Sec:StagedTrees} that the staged tree models for which $\Mcal_\T = \overline{\Mcal}_\T$ also produce the MLE after one cycle of IPS.

\subsection{Hierarchical Models}\label{ssec:hierarchical}

A hierarchical model is a log-linear models for which the independence structure between random variables can be described by a simplicial complex. We will introduce hierarchical models using the notation in \cite{sullivant2018}. They comprise another interesting class of partition models.

Let $2^{[n]}$ be the powerset of $[n]$. A
\textbf{simplicial complex} $\Gamma$ with the ground set $[n]$ is a subset of $2^{[n]}$ with the property that if
$F \in \Gamma$ and $F' \subseteq F$, then $F' \in \Gamma$. The elements of a simplicial complex are the \textbf{faces} of $\Gamma$, the inclusion-maximal faces of $\Gamma$ are called {\bf facets}, and the set of all facets of $\Gamma$ is denoted $\mathrm{facet}(\Gamma)$. A simplicial complex can be completely identified by its set of facets. So we may, for example, write $\Gamma = [12][13]$ for the simplicial complex $\Gamma = \{\emptyset, \{1\}, \{2\}, \{3\}, \{1,2\}, \{1,3\} \}$.

Up to this point we did not assume any additional structure on $ \Delta_{m-1} $, but now we will consider a collection of random variables $X_{1}, \dots, X_{\ell}$ with discrete state spaces $R_{1}, \dots, R_{\ell}$. A simplicial complex with the ground set $[\ell]$ describes the interactions among the different random variables on the joint state space $R = R_{1} \times \dots \times R_{\ell}$ . The corresponding log-linear model is called hierarchical model and defined as follows.

\begin{definition} \label{def:HierarchicalModel} Let $\Gamma$ be a simplicial complex. For every facet $F = \{f_{1}, f_{2}, \dots \}$ in $\mathrm{facet}(\Gamma)$, let $r_{F} = \prod_{f \in F} \vert R_{f} \vert$. For $i = (i_{1}, \dots, i_{\ell}) \in R$ let $i_{F}= (i_{f_{1}}, i_{f_{2}}, \dots )$. Additionally, we introduce a set of $r_{F}$ positive parameters $\theta^{F}_{i_F}$ for every facet. Now the \textbf{hierarchical model} corresponding to $\Gamma$ is 
\begin{align*}
\mathcal{M}_{\Gamma} = \left\lbrace p \text{ p.d. on } R \, \vert \, p_{i} = \dfrac{1}{Z(\theta)} \prod\limits_{F \in \mathrm{facet}(\Gamma)} \theta^{F}_{i_{F}} \right\rbrace
\end{align*}
with the normalizing factor
\begin{align*}
Z(\theta) = \sum\limits_{i \in R} \prod\limits_{F \in \mathrm{facet}(\Gamma) } \theta^{F}_{i_{F}}.
\end{align*}
\end{definition}

A description of the partition matrices that define a hierarchical model can be found in Section \ref{Sec:HierarchicalModels}. A hierarchical model with the simplicial complex $\Gamma$ can be associated with its skeleton graph $G(\Gamma)$. The vertex set of this graph is the ground set $[\ell]$ and two vertices are connected by an edge if they lie together in a face in $\Gamma$. 

\begin{definition}
A simplicial complex $\Gamma$ is called \textbf{decomposable} if it is a simplex or if $\mathrm{facet}(\Gamma)$ is the union of two decomposable simplicial complexes $\Gamma_{1}, \Gamma_{2}$ such that there exist $F_{1} \in \Gamma_{1}$ and $F_{2} \in \Gamma_{2}$ with
\begin{equation*}
\bigcup\limits_{F \in \Gamma_{1}} F \cap \bigcup\limits_{\tilde{F} \in \Gamma_{2}} \tilde{F} = F_{1} \cap F_{2}.
\end{equation*}
\end{definition}

\begin{example} \label{ex:hierModels}
\begin{itemize}
\item[•] The 2x2 independence model in Example \ref{Ex:2x2Indep} is a hierarchical model on two disjoint vertices where each random variable has state space $\{1,2\}$. This is depicted in Figure \ref{Ex:HierarchicalModels} (1). This simplicial complex is decomposable.

\item[•] The decomposable simplicial complex depicted in Figure \ref{Ex:HierarchicalModels} (2) is $\Gamma = [1,2,3][3,4,5]$. 
This graph has two cliques  $ \{ 1,2,3 \}$ and $\{3,4,5\}$, which are exactly the facets of $\Gamma$.

\item[•] In the simplicial complex $\Gamma = [12][13][23]$,  the set $\{1,2,3\}$ is not a face, hence the associated graph in Figure \ref{Ex:HierarchicalModels} (3) lacks the grayed space between the nodes. This simplicial complex is not decomposable.
\end{itemize}
\end{example}

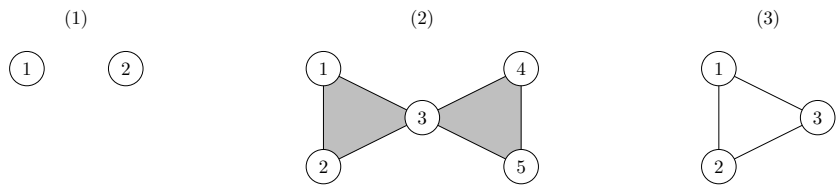
\begin{figure}[h]
\centering
\scalebox{0.65}{
\begin{tikzpicture}
\draw[] (1,1) node{(1)};
\draw[] (0,0) node {$1$};
\draw[] (0,0) circle (0.35cm);
\draw[] (2,0) node {$2$};
\draw[] (2,0) circle (0.35cm);

\draw[fill = lightgray] (6,0)--(6,-2)--(8,-1)--(6,0);
\draw[fill = lightgray] (10,0)--(10,-2)--(8,-1)--(10,0);
\draw[fill = white] (10,-2) circle (0.35cm);
\draw[fill = white] (6,0) circle (0.35cm);
\draw[fill = white] (6,-2) circle (0.35cm);
\draw[fill = white] (8,-1) circle (0.35cm);
\draw[fill = white] (10,0) circle (0.35cm);
\draw[] (10,-2) node {$5$};
\draw[] (6,0) node {$1$};
\draw[] (6,-2) node {$2$};
\draw[] (10,0) node {$4$};
\draw[] (8,-1) node {$3$};
\draw[] (8,1) node {(2)};

\draw[] (14,0)--(14,-2)--(16,-1)--(14,0);
\draw[fill = white] (14,0) circle (0.35cm);
\draw[fill = white] (14,-2) circle (0.35cm);
\draw[fill = white] (16,-1) circle (0.35cm);
\draw[] (14,0) node {$1$};
\draw[] (14,-2) node {$2$};
\draw[] (16,-1) node {$3$};
\draw[] (15,1) node{(3)};
\end{tikzpicture} }
\caption{Three different examples of hierarchical models.}\label{Ex:HierarchicalModels}
\end{figure}

In \cite{lauritzen1996} in Theorem 2.25 Lauritzen shows that every decomposable simplicial complex is conformal. The simplicial complex in Example \ref{ex:hierModels} (2) is decomposable with $F_{1} = \{ \{1,2,3\} \}$ and $F_{2} = \{ \{3,4,5\} \}$. The example in \ref{ex:hierModels} (3) is not decomposable. 
 
Now we will introduce a special ordering to the facets of a simplicial complex.
\begin{definition} \label{def:RIP}
Let $E = \{F_{1}, \dots , F_{s} \} $ be an ordering of the facets of a simplicial complex $\Gamma$. Then this ordering satisfies the \textbf{running intersection property (RIP)}, if for each $r \in [s]$ there exists a $k_{r}$ such that 
\begin{equation*}
\left( \bigcup\limits_{k = 1}^{r} F_{k} \right) \cap F_{r+1} = F_{k_{r}} \cap F_{r+1}
\end{equation*}  
\end{definition}

The connection between decomposable simplicial complexes and the running intersection property, as stated in the Lemma below, is proven for example in Lemma 5.10 in \cite{SH74}.
\begin{lemma} \label{lem:decompostoRIP}
Let $\Gamma$ be a decomposable simplicial complex. Then there exists an ordering of the facets that satisfies the RIP.
\end{lemma}

The following is Theorem 5.3 in \cite{SH74}. It draws a connection between decomposability of a simplicial complex and one-cycle convergence of the IPS algorithm for the corresponding hierarchical model. One-cycle convergence can be achieved by requiring that the \multipart matrix is ordered so that the facets corresponding to the partition matrices satisfy the RIP. An additional condition for one-cycle convergence was proven in \cite{JV99}.

\begin{theorem}[\cite{SH74}] \label{thm:RationalIfDecomposable}
Suppose the simplicial complex $\Gamma$ is decomposable.. Let $A = A^{1,\dots,k}$ be the \multipart matrix corresponding to the ordering of the facets of $\Gamma$ that satisfies the RIP. Then the iterative proportional scaling algorithm applied to $A$ yields the maximum likelihood estimate in one cycle.
\end{theorem}

In the following section, we define a generalized running intersection property that can be applied to a broader range of partition models. We show that \multipart matrices that satisfy the generalized running intersection property exhibit one-cycle convergence under iterative proportional scaling.

\section{Generalized Running Intersection Property}\label{sec:GRIP}

In this section, we define the generalized running intersection property, or GRIP.
We show that a partition model satisfies the GRIP if and only if it can be written
as an iterated toric fiber product of partition models.
This allows us to give a formula for the maximum likelihood estimate 
of a model that satisifies the GRIP.
We use this formula to show that iterative proportional scaling
exhibits one-cycle convergence on models that satisfy the GRIP.

\begin{definition}\label{def:matrixconnected}
Let $B$ and $C$ be two partition matrices with the same number of columns and with rows $\alpha^B_u$ and $\alpha^C_{u'}$.
Two rows $\alpha^B_u$ and $\alpha^C_{u'}$ are \textbf{connected} if their supports intersect nontrivially;
that is, if $I^B_u \cap I^C_{u'} \neq \emptyset$.
\end{definition}

For a matrix to satisfy the GRIP, we require that rows from one partition to the next are connected in a way that satisfies certain conditions. The following definition is motivated by the way column weights (see Definition \ref{def:colweight}) appear in the expressions produced by each step of the IPS algorithm.

\begin{definition}\label{def:wellconnected}
For a \multipart matrix $A = A^{1,\dots,k}$, define $c^\ell_j$ to be $j$-th column weight for the matrix $A^{1,\dots,\ell}$ obtained by only considering the first $\ell$ partitions of $A$. Then $A$ is \textbf{well-connected} if for any row vector $\alpha^\ell_u$ with $\ell > 1$, we have that $(c^{\ell}_j/c^{\ell-1}_j)=(c^{\ell}_{j'}/c^{\ell-1}_{j'})$ for all $j,j'\in I^\ell_u$. We call this quantity the \textbf{connection ratio} for $\alpha^\ell_u$ and denote this quantity as $C^\ell_u$ with the convention that $C^1_u =|I^1_u|$.
\end{definition}

\begin{remark}\label{Rem:ConnectionRatio}
Suppose that $A$ is a well-connected matrix. Then for any index $j \in [m]$, we see that the column weight $c_j$ can be expressed as the product of the connection ratios,
$c_j = \prod_{\ell =1}^k C^\ell_{\Scal(\ell,j)}.$
\end{remark}

\begin{example}\label{ex:connectionratio}
Consider the matrix $A$ in Figure \ref{Fig:MatrixCR}. We note that it is well-connected and display the connection ratios of each row to the right of the matrix.

$$
\begin{matrix}A & C_u^\ell\\
\begin{pmatrix}
 1 & 1 & 1 & 1 & 1 & 1 & 1  & \cdot & \cdot & \cdot & \cdot & \cdot & \cdot & \cdot\\
 \cdot & \cdot & \cdot & \cdot & \cdot & \cdot & \cdot  & 1 & 1 & 1 & 1 & 1 & 1 & 1 \\
 \hdashline
 1 & 1 & 1  & \cdot & \cdot & \cdot & \cdot & 1 & 1 & 1  & \cdot & \cdot & \cdot & \cdot\\
 \cdot & \cdot & \cdot  & 1 & 1 & 1 & 1 & \cdot & \cdot & \cdot & 1 & 1 & 1 & 1\\
 \hdashline
  1 & \cdot & \cdot & \cdot & \cdot & \cdot & \cdot & 1 & \cdot & \cdot & \cdot & \cdot & \cdot & \cdot \\
 \cdot & 1 & \cdot & \cdot & \cdot & \cdot & \cdot & \cdot & 1 & \cdot & \cdot & \cdot & \cdot & \cdot \\
 \cdot & \cdot & 1 & \cdot & \cdot & \cdot & \cdot & \cdot & \cdot & 1 & \cdot & \cdot & \cdot & \cdot \\
 \cdot & \cdot & \cdot & 1 & \cdot & 1 & \cdot & \cdot & \cdot & \cdot & 1 & \cdot & 1 & \cdot \\
 \cdot & \cdot & \cdot & \cdot & 1 & \cdot & 1 & \cdot & \cdot & \cdot & \cdot & 1 & \cdot & 1 \\
 \end{pmatrix} &\begin{matrix}
 7\\
 \\
 \frac{3}{7}\\
 \\
 \frac{1}{3}\\
 \\
 \frac{1}{3}\\
 \\
 \frac{1}{2}\\
 \end{matrix}
 \begin{matrix}
 \\
 7\\
 \\
 \frac{4}{7}\\
 \\
 \frac{1}{3}\\
 \\
 \frac{1}{2}\\
 \\
 \end{matrix}\end{matrix}
 $$
  \captionsetup{width=.8\linewidth}
 \captionof{figure}{The matrix $A$ with its connection ratios displayed on the right.} \label{Fig:MatrixCR}
\end{example}
\vspace{0.5cm}
Let $A^1, \dots, A^{\ell}$ be a set of partition matrices with $m$ columns.
Let $\beta$ be number of \emph{distinct} columns of $A^{1,\dots,\ell}$.
Define a labeling of the columns of $A^{1,\dots,\ell}$, $\lambda: [m] \rightarrow [\beta]$
such that $\lambda(j) = \lambda(j')$ if and only if the $j$th and $j'$th columns of $A^{1,\dots,\ell}$ are equal.
\begin{definition}
We define the partition matrix $B = \Cup_{n=1}^{\ell}A^n$ to be the $\beta \times m$ matrix with $j$th column $e_{\lambda(j)}$.
Since the labeling $\lambda$ of the columns of $A^{1,\dots,\ell}$ simply permutes the rows of $\Cup_{n=1}^{\ell} A^{n}$,
we omit the specification of $\lambda$ from this notation.
\end{definition}

In order to define the GRIP, we require that the tree associated to the \multipart matrix be staged.
The word  ``floret'' in the definition below  is deliberately suggestive of the terminology for staged trees discussed in Section \ref{ssec:StagedTrees}.

\begin{definition}\label{def:MatrixFloret}
Let $B$ and $C$ be two partition matrices with the same number of columns and with rows $\alpha^B_u$ and $\alpha^C_v$.
The matrices $B$ and $C$ satisfy the \textbf{floret condition} if
for every two rows of $B$, $\alpha^B_u$ and $\alpha^B_{u'}$, the sets of rows of $C$
that are connected to $\alpha^B_u$ and $\alpha^B_{u'}$ are disjoint or equal.
In this case, the set of rows of $B$ connected to a row $\alpha^C_v$
is called a \textbf{floret} of $B$ and the set of rows of $C$ connected to a row
$\alpha^B_u$ is a \textbf{floret} of $C$.
\end{definition}

The florets of the matrix $C$ correspond exactly to the second-level florets of the tree associated with the \multipart matrix $BC$. In fact, $BC$ satisfies the floret condition if and only if $\T_{BC}$ is a staged tree.

Suppose that $B$ and $C$ are two partition matrices that satisfy the floret condition,
and let $\F^C_1, \dots, \F^C_f$ be the distinct florets of $C$
with respect to the matrix $B$.
For each $t \in [f]$, let $\F^B_t$ be the set of all rows of $B$ that are connected to the rows of $\F^C_t$.

\begin{remark}
The floret condition in Definition \ref{def:MatrixFloret} implies that $\F^B_t$ is well-defined for any $t\in [f]$ and that this encompasses all florets of $B$. Indeed the florets of $B$ are in one-to-one correspondence with the florets of $C$ in this way.
\end{remark}

Let $t(u,v)$ be the index in $[f]$ such that $\alpha^{B}_u \in \F^B_{t(u,v)}$ and $\alpha^C_v \in \F^C_{t(u,v)}$.
In fact, $t(u,v)$ is determined by the row $\alpha_u^B$ \emph{or} the row $\alpha_v^C$; 
it is not necessary to provide both indices.
Hence, we write $t(u,\bullet)$ to be the index such that $\alpha_u^B \in \Fcal^B_{t(u,\bullet)}$.
Similarly, $t(\bullet,v)$ is the index such that $\alpha_v^C \in \Fcal^C_{t(\bullet,v)}$.
Thus $\alpha_u^B$ and $\alpha_v^C$ are connected if and only if $t(u,\bullet) = t(\bullet,v)$. Note that the function $t(u,v)$ implicitly depends on the matrices $B$ and $C$.

\begin{example}
Consider the matrix in Example \ref{ex:connectionratio}, and let $B=  A^{1} \Cup A^{2}$, which results in the following matrix:

$$
B = A^{1} \Cup A^{2} = \begin{pmatrix}
 1 & 1 & 1  & \cdot & \cdot & \cdot & \cdot & \cdot & \cdot & \cdot  & \cdot & \cdot & \cdot & \cdot\\
 \cdot & \cdot & \cdot  & 1 & 1 & 1 & 1 & \cdot & \cdot & \cdot & \cdot & \cdot & \cdot & \cdot\\
 \cdot & \cdot & \cdot  & \cdot & \cdot & \cdot & \cdot & 1 & 1 & 1  & \cdot & \cdot & \cdot & \cdot\\
 \cdot & \cdot & \cdot  & \cdot & \cdot & \cdot & \cdot & \cdot & \cdot & \cdot & 1 & 1 & 1 & 1\\
 \end{pmatrix}.
 $$
 Now let $C= A^3$. Note that $B$ and $C$ satisfy the floret condition and that the florets of $C$ are given by $\Fcal^C_1 = \{\alpha^C_1, \alpha^C_2, \alpha^C_3\}$ and $\Fcal^C_2 = \{\alpha^C_4, \alpha^C_5\}$.  Note that the florets of $C$ are exactly the third-level florets of the staged tree $\T_A$. We see that $t(\bullet, 1)=t(\bullet, 2)=t(\bullet, 3)=1$ and $t(\bullet, 4)=t(\bullet, 5)=2$.  The corresponding florets of $B$ are given by $\Fcal^B_1 = \{\alpha^B_1, \alpha^B_3\}$ and $\Fcal^B_2 = \{\alpha^B_2, \alpha^B_4\}$. This implies that $t(1,\bullet)=t(3,\bullet)=1$ and $t(2,\bullet)=t(4,\bullet)=2$.
At this point the function $t(u,v)$ is completely defined and it can be used to answer questions like, ``Which floret of $C$ is connected to the row $\alpha^B_3$?" This is given by $\Fcal_{t(3,\bullet)}^C = \Fcal_1^C$.
\end{example}

For a set of partition matrices $A^1, \hdots, A^\ell$ with $m$ columns, let $B = \Cup_{n=1}^{\ell} A^n$. Let $C = A^{\ell+1}$ have $\gamma$ rows and $m$ columns. 
Then the matrix $BC$ has $m$ columns of the form $[e_u \ e_v]^T$ where $u \in [\beta]$ and $v \in [\gamma]$. 
Let $\omega_{uv}$ denote the number of columns of $C$ of the form $[e_u \ e_v]^T$. 
Then the columns of $BC$ can be indexed by triples of the form $(u,v,s)$ for $s \in \{0,\dots,\omega_{uv}-1\}$
where column $(u,v,s)$ is the $(s+1)$st column of $BC$ of the form $[e_u \ e_v]^T$.
We adopt this column indexing scheme for the remainder of this section.

\begin{definition}
Suppose that $B$ and $C$ satisfy the floret condition with $f$ florets, $\Fcal^C_1,\dots,\Fcal^C_f$.
We define the matrix $B \Cap C$ to be the $f \times m$ matrix whose columns are indexed by the triples $(u,v,s)$
such that the $(u,v,s)$ entry of the $t$th row of $B \Cap C$ is equal to $1$ if $t = t(u,v)$ and $0$ otherwise. Note that any two columns with indices $(u,v,s)$ and $(u,v,s')$ are identical.
\end{definition}

In other words, a column of $B \Cap C$ are indicator vectors for the floret that the corresponding column's nonzero rows belong to in $BC$. 
Note that $B \Cap C$ is only defined when $B$ and $C$ satisfy the floret condition.

\begin{definition}\label{def:GRIP}
Let $A^1, \dots, A^k$ be partition matrices. For each $\ell$, let $B_{\ell}$ denote $\Cup_{n=1}^{\ell}A^{n}$. 
Then the \multipart matrix $A^{1,\dots,k}$ satisfies the \textbf{generalized running intersection property}, or \textbf{GRIP} if for each $\ell \in [k-1]$,
\begin{enumerate}
\item the matrix $B_{\ell}A^{\ell+1}$ is well-connected,
\item $B_{\ell}A^{\ell+1}$ satisfies the floret condition, and
\item the rows of  $B_{\ell} \Cap A^{\ell+1}$ lie in the rowspan of $A^{1,\dots,\ell}$.
\end{enumerate}
\end{definition}

\begin{remark}\label{rem:EquivGRIP}
The first point in the above definition is equivalent to $A^{1,\hdots,k}$ being well-connected. 
Additionally it is easy to see that the second point is satisfied if and only if $\T_{A^{1,\hdots,k}}$ is a staged tree. We frame these conditions above in an iterative way to facilitate the proofs in this section.
\end{remark}

\begin{example}\label{ex:GRIPtree} We show that the matrix $A$ from Example \ref{ex:connectionratio} satisfies the GRIP. Since $A$ is well-connected and $\T_A$ is a staged tree (see Example \ref{ex:firststagedtree}) $A$ satisfies conditions (1) and (2) from Definition \ref{def:GRIP}. We have included the the associated tree below, as well as all the necessary partition matrix operations to check the third condition of the GRIP.

{\centering
\begin{tikzpicture}
[grow=right, edge from parent/.style={draw,-latex}, line width = 0.75pt]

\node[circle, draw]  {}
child[level distance = 5mm, sibling distance=2.5cm, ] {node[circle,draw, fill= violet!40] {} { 
	child[sibling distance=1.25cm, level distance =5mm]{node[circle,black, draw,fill= red!50, solid] {}
		child[sibling distance=0.5cm,  level distance = 10mm]{node[circle,draw, double, solid] {}
		edge from parent[double, solid]}
		child[sibling distance=0.5cm, level distance = 10mm]{node[circle,draw, double, solid] {}
		edge from parent[double, solid]}
	edge from parent[line width = 0.75pt, black, loosely dotted]}
	child[sibling distance=1.25cm , level distance = 5mm]{node[circle,solid, draw,  fill=blue!50, line width = 0.75pt] {}
		child[sibling distance=0.5cm,  level distance = 10mm]{node[circle,draw, solid] {}
		edge from parent[line width =  0.75pt, black, solid]}
		child[sibling distance=0.5cm, level distance = 10mm]{node[circle,draw, solid] {}
		edge from parent[line width =  0.75pt, black, solid]}
		child[sibling distance=0.5cm, level distance = 10mm]{node[circle,draw, solid] {} 
		edge from parent[line width =  0.75pt, black, solid]}     
	edge from parent[line width =  0.75pt, black, densely dotted]}                                
} edge from parent[line width =  0.75pt,black, dotted]} 
child[level distance = 5mm,sibling distance=2.5cm] {node[circle,draw, fill= violet!40] {} 
		edge from parent[line width =  0.75pt, black]
	child[sibling distance=1.25cm, level distance = 5mm]{node[circle,draw,black, solid, fill=red!50] {} 
		child[sibling distance=0.5cm, level distance = 10mm]{node[circle,draw, double, solid] {}
		edge from parent[double,solid, black]}
		child[sibling distance=0.5cm, level distance = 10mm, solid]{node[circle,draw, double, solid] {}
		edge from parent[double, black,  solid]}   edge from parent[line width = 0.75pt, black, dash pattern=on 3pt off 3pt]            
	}
	child[sibling distance=1.25cm, level distance = 5mm]{node[circle,black, draw, fill=blue!50, solid]{}
		child[sibling distance=0.5cm,draw,  level distance = 10mm,solid]{node[circle,draw, solid] {} edge from parent[line width = 0.75pt, black, solid] }
		child[sibling distance=0.5cm,  level distance = 10mm]{node[circle,draw, solid] {}edge from parent[line width = 0.75pt, black, solid] }
		child[sibling distance=0.5cm,  level distance = 10mm]{node[circle,draw, solid] {} edge from parent[line width = 0.75pt, black, solid] } edge from parent[line width = 0.75pt, black, dash pattern=on 6pt off 6pt]   
	}  edge from parent[line width = 0.75pt, black, dash pattern=on 4.5pt off 4.5pt]
};
%\draw[] (-5,3) node{$A = \begin{pmatrix}
% 1 & 1 & 1 & 1 & 1 & 1 & 1  & \cdot & \cdot & \cdot & \cdot & \cdot & \cdot & \cdot\\
% \cdot & \cdot & \cdot & \cdot & \cdot & \cdot & \cdot  & 1 & 1 & 1 & 1 & 1 & 1 & 1 \\
% \hdashline
% 1 & 1 & 1  & \cdot & \cdot & \cdot & \cdot & 1 & 1 & 1  & \cdot & \cdot & \cdot & \cdot\\
% \cdot & \cdot & \cdot  & 1 & 1 & 1 & 1 & \cdot & \cdot & \cdot & 1 & 1 & 1 & 1\\
% \hdashline
%  1 & \cdot & \cdot & \cdot & \cdot & \cdot & \cdot & 1 & \cdot & \cdot & \cdot & \cdot & \cdot & \cdot \\
% \cdot & 1 & \cdot & \cdot & \cdot & \cdot & \cdot & \cdot & 1 & \cdot & \cdot & \cdot & \cdot & \cdot \\
% \cdot & \cdot & 1 & \cdot & \cdot & \cdot & \cdot & \cdot & \cdot & 1 & \cdot & \cdot & \cdot & \cdot \\
% \cdot & \cdot & \cdot & 1 & \cdot & 1 & \cdot & \cdot & \cdot & \cdot & 1 & \cdot & 1 & \cdot \\
% \cdot & \cdot & \cdot & \cdot & 1 & \cdot & 1 & \cdot & \cdot & \cdot & \cdot & 1 & \cdot & 1 \\
% \end{pmatrix} $};
 \draw[] (-5.6,2) node{$A^{1} \Cap A^{2} \  = \, \begin{pmatrix}
 1 & 1 & 1 & 1 & 1 & 1 & 1  & 1 & 1 & 1 & 1 & 1 & 1 &1 \\
 \end{pmatrix} $};
  \draw[] (-5.5,0.5) node{$A^{1} \Cup A^{2} = \begin{pmatrix}
 1 & 1 & 1  & \cdot & \cdot & \cdot & \cdot & \cdot & \cdot & \cdot  & \cdot & \cdot & \cdot & \cdot\\
 \cdot & \cdot & \cdot  & 1 & 1 & 1 & 1 & \cdot & \cdot & \cdot & \cdot & \cdot & \cdot & \cdot\\
 \cdot & \cdot & \cdot  & \cdot & \cdot & \cdot & \cdot & 1 & 1 & 1  & \cdot & \cdot & \cdot & \cdot\\
 \cdot & \cdot & \cdot  & \cdot & \cdot & \cdot & \cdot & \cdot & \cdot & \cdot & 1 & 1 & 1 & 1\\
 \end{pmatrix} $};
   \draw[] (-6,-1.5) node{$(A^{1} \Cup A^{2}) \Cap A^{3} = \begin{pmatrix}
 1 & 1 & 1  & \cdot & \cdot & \cdot & \cdot & 1 & 1 & 1  & \cdot & \cdot & \cdot & \cdot\\
 \cdot & \cdot & \cdot  & 1 & 1 & 1 & 1 & \cdot & \cdot & \cdot & 1 & 1 & 1 & 1\\
 \end{pmatrix} $};
   \draw[ dash pattern=on 3pt off 3pt] (-6.65,1.5) rectangle (-4.65,-0.5);
  \draw[ dash pattern=on 6pt off 6pt] (-8.25,1.5) rectangle (-6.75,-0.5);
 \draw[densely dotted] (-4.5,1.5) rectangle (-3,-0.5);
  \draw[loosely dotted] (-2.85,1.5) rectangle (-0.9,-0.5);
 \end{tikzpicture} 
 }
  \captionsetup{width=.8\linewidth}
 \captionof{figure}{The matrices $A^{1} \Cap A^{2}$, $A^{1} \Cup A^{2}$ and $(A^{1} \Cup A^{2}) \Cap A^{3}$ on the left and the tree associated to $A$ on the right.   }
\vspace{0.25cm}
 
For $\ell=1$, $B_1=A^{1}$ and $B_1\Cap A^2 = A^2$, so  (3) is trivially satisfied. Thus all that is left is to check that the rows of $B^2\Cap A^3$ lie in the rowspan of $A$. First consider $B^2 = A^{1} \Cup A^{2}$. The rows of this matrix are indexed by the distinct columns of the \multipart matrix $A^{1,2}$ and thus, all the distinct paths in $\T_{A^{1,2}}$. 

Note the rows of the matrix $B^2\Cap A^3=(A^{1} \Cup A^{2}) \Cap A^{3}$ index the level three florets of the staged tree $\T_A$. Since $B^2\Cap A^3 = A^2$, we can easily see that (3) is satisfied.
\end{example}

We now describe the iterated toric fiber product structure on the partition model defined by a matrix $A^{1,\dots,k}$ that satisfies the GRIP. The following lemma is crucial to our construction of the toric fiber product.

\begin{lemma}\label{lem:NumberTheory}
Suppose that $BC$ is well-connected. Then there exist positive integers $x_1, \dots, x_{\beta}, y_1, \dots, y_{\gamma}$ such that whenever $\omega_{uv}$ is nonzero, we have $\omega_{uv} = x_u y_v$ for all $u \in [\beta]$ and $v \in [\gamma]$.
\end{lemma}

\begin{proof}
By Remark \ref{Rem:ConnectionRatio}, for all $u \in [\beta], v \in [\gamma]$ such that $t(u,\bullet) = t(\bullet, v)$, we have $\omega_{uv} = C^1_u C^2_v$ where $C^1_u$ is an integer and $C^2_v$ is a rational number. Let $C^2_v = a_v/b_v$ for positive coprime integers $a_v$ and $b_v$. Since $(C^1_u a_v)/b_v$ is an integer for all $u$ and $v$, we have that $b_v$ divides $C^1_u$ for all $u$ and $v$. Let $x_u = C^1_u / \mathrm{lcm}(b_{h})_{h \in [\gamma]}$ and let $y_v = (a_v \mathrm{lcm}(b_{h})_{h \in [\gamma]})/b_v$. Each of these is an integer and satisfies $x_u y_v = C^1_u C^2_v = \omega_{uv}$. Applying this result to each of the florets yields the desired integers.
\end{proof}
%Thus by Proposition 1 in \cite{stackexchange}, there exist positive integers $x_u$, $y_v$ such that $\omega_{uv} = x_u y_v$.
%Applying this proposition to each of the florets yields the desired result.
%\end{proof}

We assume throughout that $A^{1,\dots,k}$ satisfies the GRIP. Consider the submatrix $A^{1,\dots,\ell+1}$, where $\ell \in [n-1]$. Fix $B = \Cup_{n=1}^{\ell} A^n$ and $C = A^{\ell+1}$ as in the paragraph preceding Definition \ref{def:GRIP}.

Define the polynomial ring, 
\[ R = \C[p_{u,v,s}^{t(u,v)} \mid u \in [\beta], v \in [\gamma], s \in \{0,\dots,\omega_{uv}-1\}].\]
Let $x_1,\dots, x_{\beta}, y_1,\dots, y_{\gamma}$ be positive integers such that $\omega_{uv} = x_uy_v$; these are guaranteed to exist by Lemma \ref{lem:NumberTheory}. Define two other polynomial rings, 
\begin{align*}
R_B &= \C[q_{u,s'}^{t(i,\bullet)} \mid u \in [\beta], s' \in \{0,\dots,x_u-1\}], \text{ and}\\
R_C &= \C[r_{v,s''}^{t(\bullet,v)} \mid v \in [\gamma], s'' \in \{0, \dots, y_v-1\}].
\end{align*}

Recall that $\lambda$ maps the columns of $A^{1,\dots,\ell}$ onto $[\beta]$ by labeling
the columns of $A^{1,\dots,\ell}$ so that $\lambda(u) = \lambda(v)$ if and 
only if the $u$th and $v$th columns of $A^{1,\dots,\ell}$ are equal.
Let $\tilde{A}^{1,\dots,\ell}$ be the matrix defined as follows.
Its columns are indexed by ordered pairs $(u, s')$ such that $u \in [\beta]$ and $s' \in \{0,\dots,x_u-1\}$.
The $(u,s')$ column of $\tilde{A}^{1,\dots,\ell}$ is equal to the columns of $A^{1,\dots,\ell}$ with label $u$.
In other words, the matrix $\tilde{A}^{1,\dots,\ell}$ has the same underlying set of columns as $A^{1,\dots,\ell}$;
however, the column with label $u$ is repeated in $\tilde{A}^{1,\dots,\ell}$ only $x_u$ times.
 In general, we say that a matrix $\bar{A}$ is a {\bf compression} of $A$ if $A$ and $\bar{A}$  
have the same underlying set of columns,
and $\bar{A}$ contains at most as many copies of each column as $A$.

Similarly, let $\tilde{A}^{\ell+1}$ have columns indexed by ordered pairs $(v,s'')$ such that $v \in [\gamma]$ and $s'' \in \{0, \dots, y_v -1\}$.
The $(v,s'')$ column of $\tilde{A}^{\ell+1}$ is equal to the $v$th standard basis vector. 
The matrix $\tilde{A}^{\ell+1}$ has the same underlying set of columns as $A^{\ell+1}$, 
and this underlying set is equal to $\{e_1, \dots, e_{\gamma}\}$; however, the $v$th standard basis vector
is repeated in $\tilde{A}^{\ell+1}$ only $y_v$ times.

Recall that for any integer matrix $A$, $I(A)$ is vanishing ideal of the log-linear statistical model $\Mcal_A$ as described in Proposition \ref{Prop:SameRowspan}.
We define a ring homomorphism by
\begin{alignat}{3}\label{eq:TFP}
\psi: \ & \quad R  & \rightarrow & R_B / I(\tilde{A}^{1,\dots,\ell}) \otimes_{\C} R_C / I(\tilde{A}^{\ell+1}) \\
& p_{u,v,s}^{t(u,v)} &  \mapsto & q_{u,s'}^{t(u,\bullet)} \otimes_{\C} r_{v,s''}^{t(\bullet,v)},\nonumber
\end{alignat}
where $s', s''$ are the nonnegative integers such that $s = s' y_v + s''$.
Note that each $s \in \{0,\dots, \omega_{uv}-1\}$ has a unique representation in this way. 
Define a multigrading on $R_B$ by $\deg(q_{u,s'}^{t(u,\bullet)}) = e_{t(u,\bullet)}$ and a multigrading on $R_C$ by $\deg(r_{v,s''}^{t(\bullet,v)}) = e_{t(\bullet,v)}$.
We will show that $\ker(\psi)$ is the toric fiber product of $I(\tilde{A}^{1,\dots,\ell})$ and $I(\tilde{A}^{\ell+1})$ with respect to this multigrading
and is equal to $I(A^{1,\dots,\ell+1})$.

\begin{prop}\label{prop:Multihom}
Let $\bar{A}^{1,\dots,\ell}$ and $\bar{A}^{\ell+1}$ be any compressions of $A^{1,\dots,\ell}$ and $A^{\ell}$, respectively.
The toric ideals $I(\bar{A}^{1,\dots,\ell})$ and $I(\bar{A}^{\ell+1})$ are multihomogeneous 
with respect to the multigradings $\deg(q_{u,s'}^{t(u,\bullet)}) = e_{t(u,\bullet)}$ and $\deg(r_{v,s''}^{t(\bullet,v)}) = e_{t(\bullet,v)}$.
\end{prop}

\begin{proof}
Let $D = B \Cap C$.
The rows of $D$ are in the rowspan of $A^{1,\dots,\ell}$ by the generalized running intersection property.
Moreover, the row of $D$ corresponding to a floret $\F$ in $C$ is $\displaystyle\sum_{ \alpha^C_v \in \F} \alpha^C_v$.
Hence, the rows of $D$ are in the rowspan of $A^{\ell+1}$ as well.

The matrix $\bar{A}^{1,\dots,\ell}$ is obtained by deleting some columns of $A^{1,\dots,\ell}$.
Let $\bar{D}$ be obtained from $D = B \Cap C$ by deleting the same columns.
Then the matrix $\bar{D}$ defines the multigrading $\deg(q_{u,s'}^{t(u,\bullet)}) = e_{t(u,\bullet)}$ on $R_B$.
The rows of $\bar{D}$ remain in the rowspan of $\bar{A}^{1,\dots,\ell}$.
Thus $I(\bar{A}^{1,\dots,\ell})$ is multihomogeneous with respect to the grading given by $\bar{D}$
since $\ker(\bar{A}^{1,\dots,\ell}) \subset \ker(\bar{D})$.
The argument for the linear ideal $I(\bar{A}^{\ell+1})$ is analogous.
\end{proof}

\begin{theorem}\label{thm:IdealEqualsTFP}
Let $A^{1,\dots,\ell+1}$ be a \multipart matrix that satisfies the generalized running intersection property. Then $I(A^{1,\dots,\ell+1})$ can be obtained via a toric fiber product. In particular this toric fiber product is given by $\ker(\psi)$, defined in \eqref{eq:TFP}.
\end{theorem}

\begin{proof} 
By Proposition \ref{prop:Multihom}, $I(\tilde{A}^{1,\dots,\ell})$ and $I(\tilde{A}^{\ell+1})$ are both multihomogeneous with respect to the grading specified by $D$.
For each $h \in [\ell+1]$, let $n_h$ denote the number of rows of $A^h$. 
We can replace each ideal with its parametrization to construct the map $\psi'$,

\begin{alignat*}{3}
R &\rightarrow &\C[\eta^h_g \mid h\in [\ell], g \in [n_h]] \otimes_{\C} \C[\eta^{\ell+1}_g \mid g \in [n_{\ell+1}]] \\
p_{u,v,s}^{t(u,v)} & \mapsto & \prod_{h=1}^{\ell} \eta^h_{\Scal(h,(u,v,s))} \otimes_{\C} \eta^{\ell+1}_{\Scal(\ell+1,(u,v,s))} \\
& & = \prod_{h=1}^{\ell+1} \eta^h_{\Scal(h,(u,v,s))}.
\end{alignat*}
The map $\psi'$ has the same kernel as $\psi$, which is equal to the toric fiber product.
This is exactly the map  $\phi_{A^{1,\dots,\ell+1}}$ whose kernel is $I(A^{1,\dots,\ell+1})$, as needed.
\end{proof}

%For each $u \in [\beta]$, $\F^C_{t(u,\bullet)}$ is the set of rows of $C$ that are connected to $\alpha^B_u$.
% Let $x_u, y_v $ be the positive integers such that
%\[
%\omega_{uv} = \begin{cases} x_u y_v & \text{ if } t(u,\bullet) = t(\bullet,v) \\
%0 & \text{ otherwise,}
%   \end{cases}
%\]
%which are guaranteed to exist by Lemma \ref{lem:NumberTheory}. 
We are now ready to examine the maximum likelihood estimators of \multipart matrices that satisfy the GRIP.
Recall that the number of copies of column $e_u$ in $B$ is
\[
\sum_{v:\alpha^C_v \in \F^C_{t(u,\bullet)}} \omega_{uv} = x_u \sum_{v:\alpha^C_v\in \F^C_{t(u,\bullet)}}y_v.
\]
For each $u$, let 
$$
\displaystyle Y_u = \sum_{v: \alpha^C_v \in \F^C_{t(u,\bullet)}}y_v.
$$
Then we can index columns of $A^{1,\dots,\ell}$ by $(u,s)$ where $u \in [\beta]$ and $s \in \{0,\dots,x_uY_u - 1\}$.
Let $d$ be an $m$-dimensional normalized data vector for $A^{1,\dots,\ell}$ with this indexing. 
We form an associated data vector $\tilde{d}$ for $\tilde{A}^{1,\dots,\ell}$ by setting
\begin{equation}\label{eqn:CompressedData}
\tilde{d}_{(u,s')} = \sum_{n = 0}^{Y_u-1} d_{(u,s'Y_u + n)}
%\tilde{d}_{(u,s')} = \sum_{\alpha_v^C \in \F^C_{t(u,\bullet)}} \sum_{n = 0}^{y_v-1} d_{(u,v,s'y_v + n)} 
\end{equation}
for each $u \in [\beta]$ and $s' \in \{0,\dots,x_u-1\}$.
When we compress the matrix $A^{1,\dots,\ell}$ to $\tilde{A}^{1,\dots,\ell}$, we have a single column with label $(u,s')$ representing $Y_u$ copies of an identical column in $A^{1,\dots,\ell}$. In order to form the $(u,s')$ coordinate of $\tilde{d}$, we add the $Y_u$ coordinates of $d$ that are represented by column $(u,s')$ in $\tilde{A}^{1,\dots,\ell}$.
Let $\p_{(u,s)}(d)$ denote the $(u,s)$ component of the MLE for $d$ in  $\Mcal(A^{1,\dots,\ell})$.

\begin{prop}\label{prop:CompressedMLE}
Let $d$ be a data vector for $\Mcal(A^{1,\dots,\ell})$ and let $\tilde{d}$ be as defined in Equation (\ref{eqn:CompressedData}).
Then the MLE for $\tilde{d}$ in $\Mcal(\tilde{A}^{1,\dots,\ell})$ has $(u,s')$ component $Y_u \cdot \p_{(u,0)}(d).$
\end{prop}

\begin{proof}
Let $\q$ be the $(\sum x_u)$-dimensional vector with $(u,s')$ component equal to \\ $Y_u\cdot \p_{(u,0)}(d)$ for each $u \in [\beta]$ and $s' \in \{0,\dots,x_u-1\}$. 
We show that $\q$ satisfies the conditions of Proposition \ref{prop:Birch}.

First, we show that $\tilde{A}^{1,\dots,\ell} \q = \tilde{A}^{1,\dots,\ell} \tilde{d}$. 
Let $\tilde{\alpha}$ be a row of $\tilde{A}^{1,\dots,\ell}$ and let $\alpha$ be the corresponding row in $A^{1,\dots,\ell}$.
Then we have
\[
\tilde{\alpha} \cdot \q = \sum_{(u,s') \in \supp(\tilde{\alpha})} Y_u\cdot \p_{(u,0)} (d).
\]
Note that if $(u,s') \in \supp(\tilde{\alpha})$ for some $s'$, then we have $(u,h) \in \supp(\tilde{\alpha})$ for all $h$ in $\{0,\dots,x_u -1\}$.
Moreover, by Proposition \ref{prop:Birch} applied to $\Mcal(A^{1,\dots,\ell})$, $\p_{(u,s)} = \p_{(u,0)}$ for all $s \in \{0,\dots,x_uY_u -1\}$.
Thus
\begin{align*}
\tilde{\alpha} \cdot \q &= \sum_{u: (u,0) \in \supp(\tilde{\alpha})} x_uY_u\cdot \p_{(u,0)}(d) \\
&= \sum_{u: (u,0) \in \supp(\tilde{\alpha})} \sum_{n=0}^{x_u Y_u -1} \p_{(u,n)}(d) \\
&= \alpha \cdot \p.
\end{align*}
Similarly, we have that
\begin{align*}
\tilde{\alpha} \cdot \tilde{d} &= \sum_{(u,s') \in \supp(\alpha)} \tilde{d}_{(u,s')} \\
&= \sum_{(u,s') \in \supp(\alpha)} \big( \sum_{n=0}^{x_u-1} d_{(u,s'Y_u +n)} \big) \\
&= \sum_{u:(u,0) \in \supp(\alpha)} \sum_{n=0}^{x_uY_u} d_{(u,n)} \\
&= \sum_{(u,n) \in \supp(\alpha)} d_{(u,n)} \\
&= \alpha \cdot d
\end{align*}
By Proposition \ref{prop:Birch} applied to $A^{1,\dots,\ell}$, we have $\alpha \cdot \p = \alpha \cdot d$.
Hence $\tilde{\alpha} \cdot \q = \tilde{\alpha} \cdot \tilde{d}$,
as needed.

Now we must argue that $\q$ lies in $I(\tilde{A}^{1,\dots,\ell})$. 
Let $\bar{A}^{1,\dots,\ell}$ denote the matrix obtained from $A^{1,\dots,\ell}$ by removing all repeated columns.
The columns of $\bar{A}^{1,\dots,\ell}$ are indexed by $u \in [\beta]$.
Let $J$ denote the inclusion of $I(\bar{A}^{1,\dots,\ell})$ into the polynomial ring,
\begin{align*}
\C[p_{(u,s)} \mid u \in [\beta], s \in \{0,\dots,x_uY_u -1\}]
\end{align*}
obtained
by mapping $p_u$ to $p_{(u,0)}$.
Then by definition of the toric ideal associated to a matrix, we have that
\[
I(A^{1,\dots,\ell}) = \langle p_{(u,s)} - p_{(u,r)} \mid u \in [\beta], s,r \in \{0,\dots,x_uY_u -1\} \rangle + J.
\]
Similarly, let $\tilde{J}$ denote the inclusion of $I(\bar{A}^{1,\dots,\ell})$ into $R_B$ obtained by mapping $p_u$ to $q_{(u,0)}$.
Then we have that 
\[
I(\tilde{A}^{1,\dots,\ell}) = \langle q_{(u,s)} - q_{(u,r)} \mid u \in [\beta], s,r \in \{0,\dots,x_u-1\} \rangle + \tilde{J}.
\]
Thus we must check that $\q$ lies in each of the summands of $I(\tilde{A}^{1,\dots,\ell})$.

First, $\q_{(u,s)} = \q_{(u,r)}$ for all $u\in[\beta]$ and $s,r \in \{0,\dots,x_u-1\}$ by construction.
Next, let
\[
\tilde{f} = \prod_{n=1}^N q_{(u_n,0)} - \prod_{n=1}^N q_{(v_n,0)}
\]
lie in $\tilde{J}$ and $f = \prod_{n=1}^N p_{(u_n,0)} - \prod_{n=1}^N p_{(v_n,0)}$ be the corresponding element of $J$ where $N=\deg(\tilde{f})=\deg(f)$.
By \ref{prop:Multihom}, $\tilde{J}$ is multihomogeneous with respect to the multigrading specified by $D$.
Since the columns of $D$ are linearly independent, we may assume that for each $n \in [N]$,
$\alpha_{u_n}^B$ and $\alpha_{v_n}^B$ lie in the same floret of $B$ with respect to $A^{\ell+1}$.
Thus $Y_{u_n} = Y_{v_n}$ for all $n$.
Evaluating $\tilde{f}$ at $\q$ yields
\begin{align*}
\tilde{f}(\q) = \prod_{n=1}^N \q_{(u_n,0)} - \prod_{n=1}^N \q_{(v_n,0)}  &=  \prod_{n=1}^N Y_{u_n}\cdot\p_{(u_n,0)} - \prod_{n=1}^N  Y_{v_n}\cdot \p_{(v_n,0)} \\
&= \big(\prod_{n=1}^N Y_{u_n} \big) \big( \prod_{n=1}^N \p_{(u_n,0)} - \prod_{n=1}^N \p_{(v_n,0)}\big) = 0,
\end{align*}
 since $\p \in V(J)$. 
 Thus $\q$ lies in the closure of $\Mcal(\tilde{A}^{1,\dots,\ell})$.
  So by Proposition \ref{prop:Birch} , $\q$ is the MLE for $\tilde{d}$ in $\Mcal(\tilde{A}^{1,\dots,\ell})$.
\end{proof}

Let $A$ be a \multipart matrix with $k$ partitions that satisfies the GRIP. The $\ell$th level \textbf{florets} of the matrix $A$ are given by the florets of $A^{\ell+1}$ (as defined in Definition \ref{def:MatrixFloret}) in the pair of matrices $B_\ell=\Cup_{n=1}^\ell A^n$ and $A^{\ell+1}$, with the convention that the 1st level florets are given by the rows of $A^1$. We denote the $\ell$th level floret containing the row $\alpha_i^\ell$ as $\Fcal^\ell[i]$; thus the $\ell$th level floret corresponding to the $j$th column of $A$ is $\Fcal^\ell[\Scal(\ell,j)]$.

\begin{remark}\label{rem:MatrixFloretFull}
There is a clear relationship between the notation described above and that used for florets of rows up until this point. In particular for the pair of matrices $B_\ell$ and $A^{\ell+1}$ we have that $\Fcal^{\ell+1}[\Scal(\ell+1,j)] = \Fcal^{A^{\ell+1}}_{t(\bullet,\Scal(\ell+1,j))}$ for $1\leq \ell \leq k-1$.  We make this distinction in order to avoid dependence on the matrix $B$ produced at each step. As the rows of $A$ correspond to the labels of the edges in the stratified, staged tree $\T_A$, the florets of $A$ correspond exactly to the florets of $\T_A$ (see Remark \ref{rem:stagedtreeMLE}).
\end{remark}

\begin{cor}\label{cor:GRIPMLE}
Let $A$ be a \multipart matrix with $k$ partitions that satisfies the GRIP. Then the MLE $\p$ of $d$ has as its $j$th coordinate function:
\begin{equation}\label{eq:GRIPMLE}
\p_{j} = \frac{1}{c_{j}} \left( \prod\limits_{\ell=1}^{k} \dfrac{\alpha_{\mathcal{S}(\ell,j)}^{\ell}(d)}{\sum\limits_{\alpha_i^{\ell} \in \mathcal{F}^{\ell}[\Scal(\ell,j)]} \alpha_{i}^{\ell}(d) } \right).
\end{equation}
\end{cor}

\begin{proof}
We induct on $k$. First, let $k= 1$. Then 
$I(A^1) = \langle p_j - p_{j'} \mid \lambda(j) = \lambda(j') \rangle.$
Thus, it is straightforward to check that
$\p_j = \frac{1}{c_j} \sum_{i \in I^1_{\Scal(1,j)}} d_i = \frac{1}{c_j} \alpha^1_{\Scal(1,j)} \cdot d$,
as needed.

Now suppose that the result holds for all $\ell \leq K$ for some natural number $K\geq 1$. Consider a patrition matrix $A^{1,\hdots, K+1}$ that satisfies the GRIP; then by definition of the GRIP so does $A^{1,\hdots, K}$. Now let $B = \Cup_{\ell=1}^{K} A^\ell$, $C=A^{K+1}$, and $D = B\Cap C$. By the GRIP, $BC$ satisfies the floret condition. By Lemma \ref{lem:NumberTheory},
for each $u \in [\beta]$ and $v \in [\gamma]$, there exist $x_u, y_v \in \Z_+$ such that
\[
\omega_{uv} = \begin{cases}
x_u y_v & \text{ if } t(u,\bullet) = t(\bullet,v) \\
0 & \text{ otherwise.}
\end{cases}
\]

Let $\tilde{A}^{K+1}$ have columns indexed by $(v,s'')$ where $s'' \in \{0,\dots,y_v-1\}$. 
The $(v,s'')$ column of $\tilde{A}^{K+1}$ is equal to the $v$th standard basis vector in $\Z^{\gamma}$.
For each $v \in [\gamma]$, let 
\[X_v = \sum_{u:\alpha^B_u \in \F^B_{t(\bullet,v)}} x_u
\]
so that $\sum_{u: \alpha^B_u \in \F^B_{t(\bullet,v)}} \omega_{uv} = X_v y_v$. 
Let $d^C$ be the vector of data with entries indexed by $(v,s'')$ such that
\[
d^C_{(v,s'')} = \sum_{u: \alpha^B_u \in \F^B_{t(\bullet,v)}} \sum_{n=0}^{x_u-1} d_{(u,v,s''x_u +n)}.
\]

Finally, $\tilde{D}$ be the $f \times f$ identity matrix
where $f$ is the number of distinct florets of $BC$.
Let $d^D$ be the vector of data indexed by the distinct florets of $BC$
with $t$th component
\[
d^D_t = \sum_{u,v : t(u,v) = t} \sum_{s = 0}^{\omega_{uv}-1} d_{(u,v,s)}.
\]

Let $\left(\p\right)^B$ denote the MLE for $\tilde{d}$ in $\Mcal(\tilde{A}^{1,\dots,K})$, $\left(\p\right)^C$ denote the MLE for $d^C$ in $\Mcal(\tilde{A}^{K+1})$ and $\left(\p\right)^D$ denote the MLE for $d^D$ in $\Mcal(\tilde{D})$.
By Theorem \ref{thm:IdealEqualsTFP}, $I(A^{1,\dots,K+1}) = I(\tilde{A}^{1,\dots,K}) \times_D I(\tilde{A}^{K + 1})$.
So by Theorem \ref{thm:AKK20}, we have that
\[
\p_{(u,v,s)}(d) = \frac{\left(\p\right)^B_{(u,s')} \left(\p\right)^C_{(v,s'')}}{\left(\p\right)^D_{t(u,v)}},
\]
where $s'$ and $s''$ are such that $s = s'y_v + s''$.
We compute $\left(\p\right)^B, \left(\p\right)^C$ and $\left(\p\right)^D$ to prove the desired result. Fix $u \in [\beta]$ and $v \in [\gamma]$. By induction and Proposition \ref{prop:CompressedMLE},
we have that

\begin{align*}
\left(\p\right)^B_{(u,s')} = Y_u\cdot \p_{(u,v,0)} &=  \frac{Y_u}{x_uY_u} \left( \prod\limits_{\ell=1}^{K} \dfrac{\alpha_{\Scal(\ell,(u,v,0))}^{\ell}(d)}{\sum\limits_{\alpha_{u'}^{\ell} \in \mathcal{F}^{\ell}[\Scal(\ell,(u,v,0))]} \alpha_{u'}^{\ell}(d) } \right) \\
&= \frac{1}{x_u}\left( \prod\limits_{\ell=1}^{K} \dfrac{\alpha_{\Scal(\ell,(u,v,s))}^{\ell}(d)}{\sum\limits_{\alpha_{u'}^{\ell} \in \mathcal{F}^{\ell}[\Scal(\ell,(u,v,s))]} \alpha_{u'}^{\ell}(d) } \right), 
\end{align*}
since $\Scal(\ell,(u,v,s)) = \Scal(\ell,(u,v,0))$ for all $s$. 
Since $\bar{A}^{K+1}$ is an identity matrix with repeated columns, we have that
\begin{align*}
\left(\p\right)^C_{(v,s'')} = \frac{1}{y_v} \sum_{n = 0}^{y_v -1} d^C_{(v,n)} &=
 \frac{1}{y_v} \sum_{u: \alpha^B_u \in \F^B_{t(\bullet,v)}} \sum_{n=0}^{\omega_{uv}} d_{(u,v,n)} \\
 &= \frac{1}{y_v} (\alpha_v^{K+1} d) \\
&= \frac{1}{y_v} \left(\alpha_{\Scal(K+1,(u,v,s))}^{K+1} d\right).
\end{align*}

Let $t = t(u,v)$. Finally, since $\bar{D}$ is an identity matrix, we have that
\begin{align*}
\left(\p\right)^D_t = d^D_t 
&= \sum_{u',v' : t(u',\bullet) = t(\bullet,v') = t } \sum_{s = 0}^{\omega_{u'v'}-1} d_{(u',v',s)} \\
&= \sum\limits_{\alpha_{v'}^{K+1} \in \mathcal{F}^{K+1}[\Scal(\ell,(u,v,s))]} \alpha_{v'}^{K+1} d. 
\end{align*}

Thus by Theorem 5.5 of \cite{AKK20}, the $(u,v,s)$ coordinate of the MLE for $u$ in $\Mcal(A^{1,\dots,K+1})$ is
\begin{align*}
\p_{(u,v,s)}(d) &= \frac{1}{x_u}\left( \prod\limits_{\ell=1}^{K} \dfrac{\alpha_{\Scal(\ell,(u,v,s))}^{\ell}(d)}{\sum\limits_{\alpha_{u'}^{\ell} \in \mathcal{F}^{\ell}[\Scal(\ell,(u,v,s))]} \alpha_{u'}^{\ell}(d) } \right) \\
& \qquad \qquad \times \frac{1}{y_v} (\alpha_{\Scal(K+1,(u,v,s))}^{K+1} d) \times \frac{1}{\sum\limits_{\alpha_{v'}^{K+1} \in \mathcal{F}^{K+1}[\Scal(\ell,(u,v,s))]} \alpha_{v'}^{K+1} (d)} \\
&= \frac{1}{x_u y_v} \left( \prod\limits_{\ell=1}^{K+1} \dfrac{\alpha_{\Scal(\ell,(u,v,s))}^{\ell}(d)}{\sum\limits_{\alpha_{u'}^{\ell} \in \mathcal{F}^{\ell}[\Scal(\ell,(u,v,s))]} \alpha_{u'}^{\ell}(d) } \right) \\
&= \frac{1}{\omega_{uv}}\left( \prod\limits_{\ell=1}^{K+1} \dfrac{\alpha_{\Scal(\ell,(u,v,s))}^{\ell}(d)}{\sum\limits_{\alpha_{u'}^{\ell} \in \mathcal{F}^{\ell}[\Scal(\ell,(u,v,s))]} \alpha_{u'}^{\ell}(d) } \right).
\end{align*}
This proves the result for $A^{1,\hdots,K+1}$ and we are done by induction.
\end{proof}

We now define an iterated toric fiber product of linear ideals.
These are defined so that if $I(A^{1,\dots,k})$ can be constructed as an iterated toric fiber product of models corresponding to partition matrices,
then $A^{1,\dots,k}$ satisfies the GRIP.

\begin{definition}\label{def:IteratedTFP}
Let $A^1,\dots,A^k$ be partition matrices such that $A^{1,\dots,\ell}$ and $A^{\ell+1}$ satisfy the floret condition for each $\ell \in [k-1]$.
Let $D_{\ell} = \big( \Cup_{n=1}^{\ell} A^n \big) \Cap A^{\ell+1}$ for each $\ell \in [k-1]$.
Suppose there exist compressions $\tilde{A}^{1,\dots,\ell}$ and $\tilde{A}^{\ell+1}$ 
of $A^{1,\dots,\ell}$ and $A^{\ell+1}$, respectively, such that $I(\tilde{A}^{1,\dots,\ell}) \times_{D_{\ell}} I(\tilde{A}^{\ell+1})$
is defined and equal to $I(A^{1,\dots,\ell+1})$
for each $\ell \in [k-1]$. Then the ideal $I$ is an \textbf{iterated toric fiber product of linear ideals}.
\end{definition}

\begin{prop}\label{thm:TFP2GRIP}
Let $A^{1,\dots,k}$ be a \multipart matrix such that $I(A^{1\dots,k})$ is an iterated toric fiber product of linear ideals.
Then $A^{1,\dots,k}$ is well-connected and satisfies the GRIP.
\end{prop}

\begin{proof}
Fix $\ell \in [k-1]$. Let $B = \Cup_{n=1}^{\ell} A^n$ and let $C = A^{\ell+1}$.
By Definition \ref{def:IteratedTFP}, $BC$ satisfies the floret condition.
Thus, $D_{\ell} = B \Cap C$ is well-defined.
Let $\tilde{A}^{1,\dots,\ell}$ and $\tilde{A}^{\ell+1}$ be such that $I(\tilde{A}^{1,\dots,\ell}) \times_{D_{\ell}} I(\tilde{A}^{\ell+1})$
is defined and equal to $I(A^{1,\dots,\ell+1})$.
Then since the toric fiber product is defined,
both $I(\tilde{A}^{1,\dots,\ell})$ and $I(\tilde{A}^{\ell+1})$ are multihomogeneous with respect to the multigrading
given by $D_{\ell}$. But $A^{1,\dots,\ell}$ is obtained from $\tilde{A}^{1,\dots,\ell}$ by adding repeated columns. Thus, each row of $D_{\ell}$ is contained in the orthogonal complement of the kernel of $A^{1,\dots,\ell}$, which is equal to the rowspan of $A^{1,\dots,\ell}$.
%and $\tilde{A}^{\ell+1}$, which are the rowspans of
%$\tidle{A}^{1,\dots,\ell}$ and $\tilde{A}^{\ell+1}$, respectively. 

Let $\tilde{A}^{1,\dots,\ell}$ have $\beta$ distinct columns with labels $1,\dots,\beta$. 
Let $x_u$ be the number of copies of column $u$ in $\tilde{A}^{1,\dots,\ell}$.
Let $y_v$ be the number of copies of the $v$th standard basis vector in $\tilde{A}^{\ell+1}$.
Then by the construction of the toric fiber product, there are $\omega_{uv} = x_u y_v$ columns of $A^{1,\dots,\ell+1}$
that consist of column $u$ of $\tilde{A}^{1,\dots,\ell}$ concatenated with $e_v$.

Let $u,u' \in [\beta]$ be such that $\alpha^B_u$ and $\alpha^B_{u'}$ belong to the same floret of $B$.
Let $\alpha^C_v$ belong to the corresponding floret of $C$
and let $\Fcal$ denote the set of all rows of $C$ in the same floret as $\alpha^C_v$.
Then the number of copies of column $u$ in $A^{1,\dots,\ell}$ is $x_u \sum_{v': \alpha^C_{v'} \in \Fcal} y_{v'}$.
Similarly, the number of copies of column $u'$ in $A^{1,\dots,\ell}$ is $x_{u'} \sum_{v': \alpha^C_{v'} \in \Fcal} y_{v'}$.
Thus we have
\[
\omega_{uv} / \big( x_u \sum_{v': \alpha^C_{v'} \in \Fcal} y_{v'} \big) = y_v / \big(\sum_{v': \alpha^C_{v'} \in \Fcal} y_{v'} \big)= \omega_{u'v} / \big(x_{u'} \sum_{v': \alpha^C_{v'} \in \Fcal} y_{v'}\big).
\]
Thus $BC$ is well-connected. The above argument holds for all $\ell \in [k-1]$. Thus $A^{1,\dots,k}$ satisfies the GRIP.
\end{proof}

We now show that for \multipart matrices that satisfy the GRIP, the IPS algorithm produces the MLE in exactly one cycle.
In order to understand this MLE, we must first establish the following equality.

\begin{lemma}\label{lem:GRIPMLE}
If $A$ is a \multipart matrix with $k$ partitions that satisfies the GRIP, then the following equation holds for any row vector $\alpha^k_i$ of $A^k$:
$$
\sum_{j\in I^k_i}\frac{1}{c_j^{k-1}}\left(\prod_{\ell = 1}^{k-1}\frac{\alpha^\ell_{\Scal(\ell,j)}(d)}{\ds \sum_{\alpha^\ell_{i'}\in \mathcal{F}^\ell[\Scal(\ell,j)]} \alpha^\ell_{i'}(d)}\right)= \ds C^k_i \sum_{\alpha^k_{i'} \in \mathcal{F}^k[i]} \alpha^k_{i'}(d).
$$
\end{lemma}

\begin{proof}
By Corollary \ref{cor:GRIPMLE}, the MLE of $\Mcal_A$ is given by $p=(p_1,\hdots,p_m)$ where
$$
p_j = \frac{1}{c_{j}^k} \left( \prod\limits_{\ell=0}^{k} \dfrac{\alpha_{\Scal(\ell,j)}^{\ell}(d)}{\sum\limits_{\alpha_i^{\ell} \in \mathcal{F}^{\ell}[\Scal(\ell,j)]} \alpha_{i}^{\ell}(d) } \right),
$$
for $j\in \{1,\hdots,m\}$. Since $p$ is the MLE it also satisfies $Ap=Ad$ by Proposition \ref{prop:Birch}. In particular for any row vector $\alpha^k_i$ we have that
$$
\alpha^k_i(p)=\sum_{j\in I^k_i}\frac{1}{c_j^k}\left(\prod_{\ell = 1}^{k}\frac{\alpha^\ell_{\Scal(\ell,j)}(d)}{\ds \sum_{\alpha^\ell_{i'}\in \mathcal{F}^\ell[\Scal(\ell,i)]} \alpha^\ell_{i'}(d)}\right)= \alpha^k_i(d).
$$
Note that for any $j\in I^k_i$, we have $(1/c_j^k) = (c^{k-1}_j/c^{k-1}_jc_j^k) = (1/C^k_ic^{k-1}_j)$ since $A$ is well-connected. Thus we can pull out a factor of $(1/C^k_i)$ out of each term. Additionally for any $j\in I^k_i$, we have that $\Scal(k,j)=i$ by definition, and hence the middle term of the above equation can be rewritten as
$$
\frac{\alpha^k_i(d)}{C^k_i\ds\sum_{\alpha^k_{i'} \in \mathcal{F}^k[j]} \alpha^k_{i'}(d)}\left(\sum_{j\in I^k_i}\frac{1}{c_j^{k-1}}\left(\prod_{\ell = 1}^{k-1}\frac{\alpha^\ell_{\Scal(\ell,j)}(d)}{\ds \sum_{\alpha^\ell_{i'}\in \mathcal{F}^\ell[\Scal(\ell,i)]} \alpha^\ell_{i'}(d)}\right)\right).
$$
The result immediately follows from this.
\end{proof}

\begin{theorem}\label{thm:GRIPOneCycle}
If $A$ is a \multipart matrix with $k$ partitions that satisfies the GRIP, then the IPS algorithm applied to $A$ results in the MLE after one cycle.
\end{theorem}

\begin{proof}
We proceed by induction on $k$. For $k=1$, IPS produces the MLE of $\Mcal_A$ after just one step. Indeed the information projection onto the linear family defined by $A$ in this case is simply the MLE of $\Mcal_A$  (see \eqref{eq:InfProj} and Proposition \ref{prop:Birch}). Now assume that for any \multipart matrix $A$ satisfying the GRIP where $k \leq K$ for some natural number $K\geq 1$, the IPS algorithm results in the MLE. By Corollary \ref{cor:GRIPMLE}, the MLE $p^k$ for $\Mcal_{A^{1,\dots,k}}$ is of the form in \eqref{eq:GRIPMLE} in this case.

Now let $A^{1,\hdots,K+1}$ be a \multipart matrix with $K+1$ partitions satisfying the GRIP. Note that based on the definition of the GRIP, $A^{1,\hdots,K}$ also satisfies the GRIP.

Performing the first $K$ steps of the IPS algorithm on $A$ is equivalent to one cycle on $A^{1,\hdots,K}$. Thus we have, by the induction hypothesis and Corollary \ref{cor:GRIPMLE}, that $p^{K}$ is the MLE of $\Mcal_{A^{1,\hdots,K}}$, i.e.
$$
p_j^{K} = \frac{1}{c_j^{K}}\left(\prod_{\ell = 1}^{K}\frac{\alpha^\ell_{\Scal(\ell,j)}(d)}{\ds \sum_{\alpha^\ell_i\in \mathcal{F}^\ell[\Scal(\ell,j)]} \alpha^\ell_{i}(d)}\right).
$$
Then the $(K+1)$st step of IPS scales each $p^{K}_j$ by
$$
\frac{\alpha^{K+1}_{\Scal(K+1,j)}(d)}{\ds \sum_{j'\in I^{K+1}_{\Scal(K+1,j)}}\frac{1}{c_{j'}^{K}}\left(\prod_{\ell = 1}^{K}\frac{\alpha^\ell_{\Scal(\ell,j')}(d)}{\ds \sum_{\alpha^\ell_{i}\in \mathcal{F}^\ell[\Scal(\ell,j')]} \alpha^\ell_{i}(d)}\right)}=\frac{\alpha^{K+1}_{\Scal(K+1,j)}(d)}{\ds C^k_{\Scal(K+1,j)} \sum_{\alpha^{K+1}_{i} \in \mathcal{F}^{K+1}[\Scal(K+1,j)]} \alpha^{K+1}_{i}(d)},
$$
where equality follows from Lemma \ref{lem:GRIPMLE}. Since $C^{K+1}_{\Scal(K+1,j)} = \frac{c_j^{K+1}}{c^{K}_j}$ this implies that $p^{K+1}$ is exactly the MLE of $\Mcal_{A^{1,\hdots,K+1}}$ by Proposition \ref{cor:GRIPMLE}. By induction this proves the result.
\end{proof}

%\begin{remark}
In order to prove one-cycle convergence for models satisfying the GRIP in Theorem \ref{thm:GRIPOneCycle}, we took three important steps, which are similar to the structure of the proof of one-cycle convergence in the case of the RIP in \cite{SH74} Theorem 5.3. There the author first proves  in Lemma 5.8 that the MLE can be written as a normalized product of the MLE for two submodels. We use a generalization of this result from \cite{AKK20} in Lemma \ref{lem:GRIPMLE} to show the structure of the MLE. This corresponds to Theorem 5.1 in \cite{SH74}. Finally, in Theorem \ref{thm:GRIPOneCycle} and Theorem 5.3 in \cite{SH74}, respectively, the now known structure of the MLE is used in the induction to show one-cycle convergence.
%\end{remark}

Not only does the proof of Theorem \ref{thm:GRIPOneCycle} show that the IPS algorithm constructs the MLE in one cycle, but also that at the $\ell$-th step the vector $p^\ell$ is exactly the MLE of the $A^{1,\hdots,\ell}$. It is natural to ask if the property of producing the MLE in one cycle also characterizes \multipart matrices satisfying the GRIP.
However, we can immediately see that this is not the case in general. Consider a $n\times m$ \multipart matrix $A$ with $k$ partitions such that $A^k$ is the identity matrix $I_m$. Then $\Mcal_A$ is the entire simplex $\Delta_{m-1}$ and IPS automatically produces vector $d$ (the MLE of $\Mcal_A$ in this case) regardless of the first $k-1$ partitions. 

\begin{remark}\label{rem:Horn}
For a \multipart matrix $A$ with $k=2$, it is possible to show that if the IPS algorithm produces the MLE in one cycle, then $A$ satisfies the GRIP. This relies on being able to explicitly write the coordinate functions of the MLE as a product of linear forms that must satisfy the Horn uniformization \cite{DMS19, HS14}. While the above counter-example implies that this is not the case for $k=3$, it would be interesting to investigate whether requiring IPS to produce the MLE of $A^{1,\hdots, \ell}$ at the $\ell$-th step is sufficient to guarantee that $A$ satisfies the GRIP.
\end{remark}

\begin{example}
Here we apply IPS to the matrix $A$ from Example \ref{Ex:Partition}. The rows of $A$ are colored blue in this Example.

Let $d = (d_{1}, \dots, d_{14})$ be the normalized data vector. Projecting to the first partition of $A$ leads to
$p^{0}_{1} = \dots = p^{0}_{7} = \dfrac{1}{7} \, { \color{blue} \alpha^{1}_{1}( d ) }$ and $p^{0}_{8} = \dots = p^{0}_{14} = \dfrac{1}{7} \, { \color{blue} \alpha^{1}_{2}(  d)}.$
The second step of the algorithm results in four different types of indices. These are given by the different rows in the matrix  $A^{1} \Cup A^{2} $, indicated by the different dashed and dotted lines below the matrix.
{\begin{center}
\begin{tikzpicture}
[grow=right, edge from parent/.style={draw,-latex}, line width = 1pt]
 \draw[] (-5.6,0) node{ \footnotesize $A^{1} \Cap A^{2} \  = \, \begin{pmatrix}
 1 & 1 & 1 & 1 & 1 & 1 & 1  & 1 & 1 & 1 & 1 & 1 & 1 &1 \\
 \end{pmatrix} $};
  \draw[] (-5.5,-1.25) node{ \footnotesize $A^{1} \Cup A^{2} = \begin{pmatrix}
 1 & 1 & 1  & \cdot & \cdot & \cdot & \cdot & \cdot & \cdot & \cdot  & \cdot & \cdot & \cdot & \cdot\\
 \cdot & \cdot & \cdot  & 1 & 1 & 1 & 1 & \cdot & \cdot & \cdot & \cdot & \cdot & \cdot & \cdot\\
 \cdot & \cdot & \cdot  & \cdot & \cdot & \cdot & \cdot & 1 & 1 & 1  & \cdot & \cdot & \cdot & \cdot\\
 \cdot & \cdot & \cdot  & \cdot & \cdot & \cdot & \cdot & \cdot & \cdot & \cdot & 1 & 1 & 1 & 1\\
 \end{pmatrix} $};
%   \draw[ dash pattern=on 3pt off 3pt] (-6.65,-0.4) rectangle (-4.65,-2.2);
%  \draw[ dash pattern=on 6pt off 6pt] (-8.25,-0.4) rectangle (-6.75,-2.2);
%\draw[densely dotted] (-4.5,-0.4) rectangle (-3,-2.2);
%\draw[loosely dotted] (-2.85,-0.4) rectangle (-0.9,-2.2);
\draw[ dash pattern=on 6pt off 6pt] (-8.1,-2.15)--(-6.85,-2.15);
\draw[ dash pattern=on 3pt off 3pt] (-6.6,-2.15)--(-4.75,-2.15);
\draw[densely dotted] (-4.5,-2.15)--(-3.25,-2.15);
\draw[loosely dotted] (-2.9,-2.15)--(-1.2,-2.15);
\draw[] (0.5,0) node{ \color{red} \tiny $\alpha^{1 \Cap 2}_{1}$};
\draw[] (0.5,-0.6) node{ \color{darkgray} \tiny $\alpha^{1 \Cup 2}_{1}$};
\draw[] (0.5,-1) node{ \color{darkgray} \tiny $\alpha^{1 \Cup 2}_{2}$};
\draw[] (0.5,-1.4) node{ \color{darkgray} \tiny $\alpha^{1 \Cup 2}_{3}$};
\draw[] (0.5,-1.8) node{ \color{darkgray} \tiny $\alpha^{1 \Cup 2}_{4}$};
\end{tikzpicture}
\end{center}
}
Note that since $A^{1} \Cup A^{2} $ consists only of the ones vector and d is normalized, we have ${ \color{red} \alpha_{1}^{1 \Cap {2}} ( d)  } = 1$. Therefore the second projection simplifies in the following way:

\begin{center}
 
\begin{tikzpicture}[line width = 1pt]
\draw[] (0,0.3) node{ \small $\begin{aligned}
p^{1}_{1} &= p^{1}_{2} = p^{1}_{3} = \quad \quad \quad \,  \frac{1}{C^{1}_{1}} {\color{blue} \alpha^{1}_{1} (d)} \dfrac{ { \color{blue}  \alpha^{2}_{1} ( d) }}{C^{2}_{1} \, { \color{red} \alpha_{1}^{1 \Cap {2}} ( d)  }} =  \frac{1}{3} {  \alpha^{1}_{1} (d) } \alpha^{2}_{1} ( d) \\
p^{1}_{4} &= p^{1}_{5} = p^{1}_{6} = p^{1}_{7} 
= \quad \frac{1}{C^{1}_{1}} {  \color{blue} \alpha^{1}_{1} (d ) } \dfrac{ \color{blue} \alpha^{2}_{2} ( d)}{C^{2}_{2} \, { \color{red} \alpha_{1}^{{1} \Cap {2}} (d) }}   = \frac{1}{4} \alpha^{1}_{1} (d ) \alpha^{2}_{2} ( d)  \\
p^{1}_{8} &= p^{1}_{9} = p^{1}_{10} = \quad \quad \quad \frac{1}{C^{1}_{2}} {\color{blue} \alpha^{1}_{2} ( d) } \dfrac{ \color{blue} \alpha^{2}_{1}( d)}{C^{2}_{1} \, { \color{red} \alpha_{1}^{{1} \Cap {2}} ( d) }} = \frac{1}{3} \alpha^{1}_{2} ( d) \alpha^{2}_{1}( d) \\
p^{1}_{11} &= p^{1}_{12} = p^{1}_{13} = p^{1}_{14} = \frac{1}{C^{1}_{2}} { \color{blue} \alpha^{1}_{2} ( d ) } \dfrac{ \color{blue} \alpha^{2}_{2} ( d)}{C^{2}_{2} \, {\color{red} \alpha_{1}^{{1} \Cap {2}} ( d ) }} =  \frac{1}{4}  \alpha^{1}_{2} ( d ) \alpha^{2}_{2} ( d) \\
\end{aligned}$};
\draw[ dash pattern=on 6pt off 6pt] (-5.5,2.15) -- (-5.5,1.5);
\draw[ dash pattern=on 3pt off 3pt] (-5.5,1.25) -- (-5.5,0.5);
\draw[densely dotted] (-5.5,0.25) -- (-5.5,-0.5);
\draw[loosely dotted] (-5.5,-0.75) -- (-5.5,-1.5);
\end{tikzpicture}
\end{center}

For the last projection we only demonstrate four indices as an example.
{\begin{center}
\begin{tikzpicture}[grow=right, edge from parent/.style={draw,-latex}, line width = 1pt]
   \draw[] (-6,-3.5) node{ \footnotesize $(A^{1} \Cup A^{2}) \Cap A^{3} = \begin{pmatrix}
 1 & 1 & 1  & \cdot & \cdot & \cdot & \cdot & 1 & 1 & 1  & \cdot & \cdot & \cdot & \cdot\\
 \cdot & \cdot & \cdot  & 1 & 1 & 1 & 1 & \cdot & \cdot & \cdot & 1 & 1 & 1 & 1\\
 \end{pmatrix} $};
\draw[] (0.85,-3.2) node{ \color{olive} \tiny $\alpha^{(1 \Cup 2) \Cap 3}_{1}$};
\draw[] (0.85,-3.65) node{ \color{olive} \tiny $\alpha^{(1 \Cup 2) \Cap 3}_{2}$};
 \end{tikzpicture} 
 \end{center}
 }

In this case $A^{2} = (A^{1} \Cup A^{2}) \Cap A^{3}$, hence ${ \color{blue} \alpha^{2}_{i} ( d) } = { \color{olive} \alpha_{i}^{({1} \Cup {2}) \Cap 3} (d)}$. In order to visualize the structure created by IPS, we also display the general formula of $p^{2}_{j}$ without simplifications  

\small 
\begin{align*}
p^{2}_{1} &= \frac{1}{C^{1}_{1}} { \color{blue} \alpha^{1}_{1} ( d ) } \dfrac{{ \color{blue} \alpha^{2}_{1} ( d) }}{C^{2}_{1} \,{\color{red} \alpha_{1}^{{1} \Cap {2}} (d) }} \dfrac{\color{blue} \alpha^{3}_{1} (d)}{C^{3}_{1} {\color{olive} \alpha_{1}^{(1 \Cup {2}) \Cap {3}} ( d) }} =  \frac{1}{7} \alpha^{1}_{1} ( d ) \dfrac{\alpha^{2}_{1} ( d)}{\frac{3}{7} \, \alpha_{1}^{{1} \Cap {2}} (d)} \dfrac{\alpha^{3}_{1} (d)}{\frac{1}{3} \alpha_{1}^{(1 \Cup {2}) \Cap {3}} ( d)} =  \alpha^{1}_{1} ( d )  \alpha^{3}_{1} ( d )\\
p^{2}_{4} &= 
\frac{1}{C^{1}_{1}} {\color{blue} \alpha^{1}_{1} (d ) } \dfrac{\color{blue} \alpha^{2}_{2} ( d)}{C^{2}_{2} \, {\color{red} \alpha_{1}^{{1} \Cap {2}} (d)}}\dfrac{\color{blue} \alpha^{3}_{4}(d)}{C^{3}_{4} { \color{olive} \alpha_{2}^{(1 \Cup 2) \Cap 3} (d)} } = \frac{1}{7} \alpha^{1}_{1} (d ) \dfrac{\alpha^{2}_{2} ( d)}{\frac{4}{7} \, \alpha_{1}^{{1} \Cap {2}} (d)}\dfrac{\alpha^{3}_{4}}{\frac{2}{4}  \alpha_{2}^{(1 \Cup 2) \Cap 3} (d) } = \frac{1}{2} \alpha^{1}_{1}(d) \alpha^{3}_{4}(d) \\
p^{2}_{8} &= 
\frac{1}{C^{1}_{2}} {\color{blue} \alpha^{1}_{2} (d ) } \dfrac{\color{blue} \alpha^{2}_{1} ( d)}{C^{2}_{1} \, {\color{red} \alpha_{1}^{{1} \Cap {2}} (d)}}\dfrac{\color{blue} \alpha^{3}_{1}(d)}{C^{3}_{1} { \color{olive} \alpha_{1}^{(1 \Cup 2) \Cap 3} (d)} } =  \frac{1}{7} \alpha^{1}_{2} ( d ) \dfrac{\alpha^{2}_{1} ( d)}{\frac{3}{7} \, \alpha_{1}^{{1} \Cap {2}} (d)} \dfrac{\alpha^{3}_{1} (d)}{\frac{1}{3} \alpha_{1}^{(1 \Cup {2}) \Cap {3}} ( d)} =  \alpha^{1}_{2} ( d )  \alpha^{3}_{1} ( d ) \\
p^{2}_{11} &= 
\frac{1}{C^{1}_{2}} {\color{blue} \alpha^{1}_{2} (d ) } \dfrac{\color{blue} \alpha^{2}_{2} ( d)}{C^{2}_{2} \, {\color{red} \alpha_{1}^{{1} \Cap {2}} (d)}}\dfrac{\color{blue} \alpha^{3}_{4}(d)}{C^{3}_{4} { \color{olive} \alpha_{2}^{(1 \Cup 2) \Cap 3} (d)} } = \frac{1}{7} \alpha^{1}_{2} (d ) \dfrac{\alpha^{2}_{2} ( d)}{\frac{4}{7} \, \alpha_{1}^{{1} \Cap {2}} (d)}\dfrac{\alpha^{3}_{4}}{\frac{2}{4} \alpha_{2}^{(1 \Cup 2) \Cap 3} (d) } = \frac{1}{2} \alpha^{1}_{1}(d) \alpha^{3}_{4}(d) 
\end{align*}
There exist two columns identical to column $4$ and two columns identical to column $11$. Hence the corresponding connection ratios do not add up to one, but to $\frac{1}{2}$.
\end{example}

We end this section by noting that the formula for the MLE of a matrix satisfying the GRIP in \eqref{eq:GRIPMLE} factors through the associated monomial map $\phi_A$. Re-ordering if necessary, we can denote $d$ in the ambient space of $\Mcal_A$ as
$d = (d_1^1, \hdots, d_1^{c_1},d_2^1,\hdots, d_2^{c_2},\hdots, d_m^{c_m})$
where $c_j$ is the column weight. We have that
$$
p^\star(d) = \phi_A\left(\frac{s_1^1(d)}{C_1^1}, \hdots, \frac{s_{n_1}^1(d)}{C_{n_1}^1}, \frac{s_1^2(d)}{C_1^2}, \hdots, \frac{s_{n_k}^k(d)}{C_{n_k}^k}\right)
$$
where $s_i^\ell(d) = \frac{\alpha^\ell_i(d)}{\displaystyle\sum_{\alpha_{i'}^\ell \in \F^\ell[i]}a_i^\ell(d)}$. This factorization implies a nice correspondence between the MLE of the matrix $\bar{A}$ obtained by removing all repeated columns of $A$ and the MLE of $A$.  Suppose that $\bar{A}$ is of size $n\times m$ and let $\bar{d} = (\bar{d}_1,\hdots,\bar{d}_m)$ be a normalized data vector of counts. 
%Re-ordering if necessary, we can denote $d$ in the ambient space of $\Mcal_A$ as
%$$
%d = (d_1^1, \hdots, d_1^{c_1},d_2^1,\hdots, d_2^{c_2},\hdots, d_m^{c_m})
%$$
%where $c_j$ is the column weight, or number of repetitions of the column, of the $j$th column of $\bar{A}$ (see Definition \ref{def:colweight}).

\begin{prop}\label{prop:RemoveRepeatMLE}
If $A$ is a \multipart matrix with $k$ partitions that satisfies the GRIP, then the MLE of $\Mcal_{\bar{A}}$ for a data vector $\bar{d}$ is given by
$$
\p_{\bar{A}}(\bar{d}) = \phi_{\bar{A}}\left(s_1^1(\bar{d}), \hdots, s^1_{n_1}(\bar{d}), s_1^2(\bar{d}),\hdots, s^k_{n_k}(\bar{d})\right)
$$

where $s_i^\ell(\bar{d}) = \frac{\bar{\alpha}^\ell_i(\bar{d})}{\displaystyle\sum_{\bar{\alpha}_{i'}^\ell \in \F[i]^\ell}\bar{\alpha}_{i'}^\ell(\bar{d})}$ for row vectors $\bar{\alpha}^\ell_i$ of $\bar{A}$.
\end{prop}

\begin{proof}
Clearly $\p_{\bar{A}}(\bar{d})$ lies on the model $\Mcal_{\bar{A}}$; it remains to show that $\bar{A}(\p_{\bar{A}}(\bar{d})) = \bar{A}(\bar{d})$ which implies that $\p_{\bar{A}}(\bar{d})$ is the MLE of $\Mcal_{\bar{A}}$ by Proposition \ref{prop:Birch}. Consider a data vector $d$ in the ambient space of $\Mcal_A$ where $d_j^i = (\bar{d}_j / c_j)$ for all $j \in [m]$ and $i \in [c_j]$. Then we have that
$\bar{d}_j = d_j^1 + d_j^1 + \hdots + d_j^{c_j}.$

Since $A$ satisfies the GRIP, the MLE of $\Mcal_A$ with respect to $d$, denoted $\p_A(d)$, is given in \eqref{eq:GRIPMLE}. For any row $\bar{\alpha}^\ell_i$ of $\bar{A}$, it is easy to see that
$\alpha^\ell_i(d) = \bar{\alpha}^\ell_i(\bar{d}).$ Thus we can write the coordinate functions of $\p_A(d)$ in terms of the coordinate functions of $\p_{\bar{A}}(\bar{d})$ as
$$
\p_A(d)_j^i  = \frac{\p_{\bar{A}}(\bar{d})_j}{c_j}
$$
%$$
%\p_A(d)_j^i = \frac{ \displaystyle\prod_{\ell=1}^k s^\ell_{\Scal(\ell,j)}(d)}{c_j} = \frac{\displaystyle\prod_{\ell=1}^k s^\ell_{\Scal(\ell,j)}(\bar{d})|_{\bar{d}_j = d_j^1+d_j^2+\hdots+d_j^{c_j}}}{c_j} = \frac{\p_{\bar{A}}(\bar{d})_j}{c_j}
%$$
%where $\p_A(d)_j^i$, $1\leq i \leq c_j$, is any coordinate function of $\p_A(d)$ corresponding to the $j$-th column of $\bar{A}$. 
for all $i \in [c_j]$. 
So we have that
$$
\sum_{i=1}^{c_j} \p_A(d)^i_j = p_{\bar{A}}(\bar{d})_j,
$$
which implies that
$$
\alpha^\ell_i(\p_A(d)) = \sum_{j=1}^m \sum_{i=1}^{c_j}  \p_A(d)^i_j  =  \sum_{j=1}^m \p_{\bar{A}}(\bar{d})_j=\bar{\alpha}^\ell_i(\p_{\bar{A}}(\bar{d})).
$$

Putting the above two equalities together, along with the fact that $\p_A(d)$ is the MLE of $\Mcal_A$, we have that
$\bar{\alpha}^\ell_i(\bar{d}) = \alpha^\ell_i(d) = \alpha^\ell_i(p_A(d)) = \bar{\alpha}^\ell_i(p_{\bar{A}}(\bar{d})),$
which in turn implies that $\bar{\alpha}^\ell_i(\bar{d}) =  \bar{\alpha}^\ell_i(p_{\bar{A}}(\bar{d}))$.
\end{proof}

\begin{example}
In Proposition \ref{prop:RemoveRepeatMLE}, if one removes all repeated columns from a \multipart matrix that satisfies the GRIP, the result also satisfies the GRIP. The reverse is, however, not true. Consider the independence model from Example \ref{Ex:2x2Indep}. We see in that example that the IPS algorithm produces the MLE in one cycle. One can easily check that this model satisfies the GRIP as well, so this is not surprising.

\noindent
\begin{minipage}[t]{0.75\textwidth}
However, if we simply repeat the last column to obtain a new \multipart matrix, shown here to the right, then the IPS algorithm does not produce the MLE after one cycle. In fact, one can check (by computing the solution set to the polynomial equations defining the ML-degree) that the resulting partition model has ML-degree greater than one, and hence non-rational MLE. This implies that the IPS algorithm \textit{cannot} produce the MLE exactly.
\end{minipage} \hfill
\begin{minipage}[t]{0.225\textwidth}
\begin{align*}
\begin{pmatrix}
1 & 1  & \cdot & \cdot   &\cdot\\
\cdot & \cdot & 1& 1 &1\\
\hdashline
1 & \cdot & 1 & \cdot &\cdot\\
\cdot & 1 & \cdot & 1 &1
\end{pmatrix}
\end{align*} 
\end{minipage}
\end{example}

\section{Model families satisfying the GRIP}\label{Sec:models}

\subsection{Hierarchical models} \label{Sec:HierarchicalModels}

In this section, we justify our use of the name ``generalized running intersection property" in the case of binary hierarchical models. 
As discussed in Section \ref{ssec:hierarchical}, we can associate a hierarchical model with a simplicial complex $\Gamma$.
Recall that if a simplicial complex is decomposable, then by Lemma \ref{lem:decompostoRIP}, there exists an ordering on the facets that satisfies RIP.
First we define the matrix $A_{\Gamma}$ that realizes the hierarchical model on $\Gamma$ as a log-linear model in the case where each random variable has exactly two states. The columns of the matrix $A_{\Gamma}$ are indexed by subsets of $[n]$.
We could equivalently index the columns by 0/1 strings of length $n$ by
taking each $S \subset [n]$ to be the the set of positions in the string equal to 1.

Let $F_1, \dots, F_k$ be the facets of $\Gamma$. The rows of $A_{\Gamma}$ are divided into blocks $A^1, \dots, A^k$ each with $2^{\vert F_i \vert}$ rows. 
The rows of the partition matrix $A^i$ are of the form $a^i_S$ where $S \subset F_i$ with
\[
a^i_S(T) = \begin{cases}
1, & \text{ if } T \cap F_i = S \\
0, &\text{ otherwise,}
\end{cases}
\]
for each $T \subset [n]$.
The {\bf hierarchical model} $\Mcal_{\Gamma}$ is the closure of the  
image of the monomial map defined by $A_{\Gamma}$ intersected with the probability simplex, $\Delta_{2^n}$.

\begin{prop}\label{prop:rip2grip}
Let $\Gamma$ be a simplicial complex and let $A$ be a \multipart matrix representing the binary hierarchical model on $\Gamma$ that satisfies the running intersection property. Then $A$ satisfies the generalized running intersection property as well.
\end{prop}

%\begin{remark}\label{rem:rip2grip}
%This proposition holds for non-binary hierarchical models as well. For the sake of notational simplicity, we provide the argument for binary hierarchical models below; however, the analogous argument proves the result for non-binary random variables as well.
%\end{remark}

\begin{proof}
Let $\Gamma$ be a simplicial complex on $[p]$ with facets $F_1,\dots,F_k$ such that the resulting \multipart matrix
$A^{1,\dots,k}$ satisfies the running intersection property. Let $1 \leq \ell \leq k$, $B = \Cup_{n=1}^{\ell} A^n$ and let $C = A^{\ell+1}$.
Without loss of generality, let $F_1,\dots,F_{\ell}$ form a simplicial complex on $[q]$ for $q \leq p$.
Since $A^{1,\dots,k}$ satisfies the RIP, there exists an $s \leq \ell$ such that
$[q] \cap F_{\ell+1} = F_s \cap F_{\ell+1}$.

First, we wish to show that the matrix $BC$ satisfies the floret condition.
The matrix $B$ has rows indexed by subsets of $[q]$. For each $S \subset [p]$ and $T \subset [q]$,
the $S$ entry of $\alpha^B_T$ is equal to $1$ if $T = S \cap [q]$ and $0$ otherwise.
Let $U \subset F_{\ell+1}$. We claim that $\alpha^C_U$ is connected to $\alpha^B_T$ if and only if
$U \cap [q] = T \cap F_{\ell+1}$. 

Indeed, suppose that $U \cap [q] = T \cap F_{\ell +1}$. Let $S = U \cap T$. Then
\[
S \cap [q] = (U \cap [q]) \cup (T \cap [q]) = (T \cap F_{\ell+1}) \cup T = T.
\]
Similarly, $S \cap F_{\ell+1} = U$. Hence the $S$ coordinates of $\alpha^B_T$ and $\alpha^C_U$ are both equal to one, and these rows are connected.

Conversely, suppose that $\alpha^B_T$ and $\alpha^C_U$ are connected. Then there exists an $S \subset [p]$ such that $U = S \cap F_{\ell+1}$ and $T = S \cap [q]$. Intersecting the first equation with $[q]$ and the second with $F_{\ell+1}$ yields that $U \cap [q] = T \cap F_{\ell+1}$, as needed.

Thus, for each $T, T' \subset [q]$, if $T \cap F_{\ell +1} = T' \cap F_{\ell+1}$, then we have that the sets of rows of $C$ connected to $\alpha^B_T$ and $\alpha^B_{T'}$ are equal. Otherwise these sets are disjoint. Hence $BC$ satisfies the floret condition.

For each $S \subset [p]$, the number of columns of $B$ identical to the column associated to $S$ is $2^{p-q}$. Indeed, the columns of $B$ associated to $S$ and $S'$ are equal if and only if $S \cap [q] = S' \cap [q]$, and there are $2^{p-q}$ such subsets of $[p]$. Similarly, the number of columns of $BC$ identical to the column associated to $S$ is $2^{p-q'}$, where $q' = \#([q] \cup F_{\ell+1})$. Thus, since all column weights within $B$ are equal to one another and all column weights within $BC$ are equal to one another, $BC$ is well-connected.

Finally, we must show that $B \Cap C$ lies in the rowspan of $A^{1,\dots,\ell}$.
Recall that there exists an $s \leq \ell$ such that $[q] \cap F_{\ell+1} = F_s \cap F_{\ell+1}$.
We showed above that the rows of $D = B \Cap C$ are indexed by subsets of $[q] \cap F_{\ell+1}$.
For each $T \subset [q] \cap F_{\ell+1}$ and each $S \subset [p]$, the row $\alpha^D_T$ has $S$ entry
equal to $1$ if $S \cap ([q] \cap F_{\ell+1}) = T$ and $0$ otherwise.

We claim that for each $T \subset [q] \cap F_{\ell+1}$, $\alpha^D_T$ is the sum of all rows of $A^s$, $\alpha^s_U$, such that $U \cap F_{\ell+1} = T$, i.e.  $\alpha^D_T$ is equal to
\begin{equation}\label{eq:HierarchicalRowspan}
\sum_{\substack{U \subset F_s: \\ U \cap F_{\ell+1} = T}} \alpha^s_U.
\end{equation}

Indeed, let $S \subset [p]$ satisfy $S \cap F_{\ell+1} \cap [q] = T$.
Then we have that $(S \cap F_s) \cap F_{\ell+1} = T$.
In particular, the term of Equation \ref{eq:HierarchicalRowspan} corresponding to $U = S \cap F_s$
has $S$ entry equal to $1$, as needed.

Similarly, if $S \cap F_{\ell+1} \cap [q] = T' \neq T$, then we have $T' = (S \cap F_s) \cap F_{\ell+1}$.
In particular, the row $\alpha^s_{S \cap F_s}$ is the row of $A^s$ supported on $S$.
But it is not a term of the sum in Equation (\ref{eq:HierarchicalRowspan}). 
Hence the sum in (\ref{eq:HierarchicalRowspan}) has $S$ entry equal to $0$, as needed.

Since each row of $D$ lies in the rowspan of $A^s$, it lies in the rowspan of $A^{1,\dots,\ell}$, as needed.
Hence $A^{1,\dots,k}$ satisfies the generalized running intersection properety.
\end{proof}

We note that an analogous argument proves the result for non-binray random variables as well; we do not include a proof here to maintain notational simplicity.

\subsection{Staged tree models}\label{Sec:StagedTrees}

In this section, we explore the implications of the results in Section \ref{sec:GRIP} for staged tree models. We show that a staged tree model is balanced and stratified if and only if the associated matrix satisfies the GRIP, leading to a new characterization of balanced, stratified staged trees. In particular for the associated \multipart matrix, the IPS algorithm results in the MLE after one cycle.

To define the notion of a balanced staged tree, first we must describe the interpolating polynomial of a staged tree $\T$, first introduced in \cite{GS18}. For any vertex $v\in V$, we can consider the subtree $\T_v$ rooted at $v$. We denote the set of root-to-leaf paths of $\T_v$ (or equivalently the $v$-to-leaf paths in $\T$) as $\Lambda_v$.

\begin{definition}
For a staged tree $(\T,\theta)$ and a vertex $v\in V$, let
$t(v) := \sum_{\lambda \in \Lambda_v} \prod_{s_i \in \theta(\lambda)} s_i.$
If $v_0$ is the root of $\T$, then $t(v_0)$ is the \textbf{interpolating polynomial} of $\T$.
\end{definition}

\begin{definition}
A staged tree $(\T,\theta)$ is \textbf{balanced} if for any two vertices $v,w \in V$ such that $\mathcal{F}_v=\mathcal{F}_w$ the following equation
\begin{equation}\label{eq:balancedcondition}
t(v')t(w'')=t(w')t(v'')
\end{equation}
is satisfied for all distinct pairs of vertices $v', v''\in \ch(v)$ and $w', w''\in \ch(w)$ where $\theta(v,v')=\theta(w,w')$ and $\theta(v,v'')=\theta(w,w'')$.
\end{definition} 

\begin{example}\label{ex:unbalanced}
Consider the staged tree $(\T,\theta)$ in Example \ref{ex:firststagedtree}. To determine if the tree is balanced we must compare all pairs of vertices in the same stage. The pairs $v_{s_0t_0}, v_{s_1t_0}$ and $v_{s_0t_1}, v_{s_1t_1}$ trivially satisfy the condition in \eqref{eq:balancedcondition}. Thus to determine whether $\T$ is balanced it is enough to check whether \eqref{eq:balancedcondition} holds for $v_{s_0}$ and $v_{s_1}$. We have that
$\ch(v_{s_0}) = \{ v_{s_0t_0}, v_{s_0t_1} \}$ and $\ch(v_{s_1}) = \{ v_{s_1t_0}, v_{s_1t_1}\}.$
So we have to check that $t(v_{s_0t_0})t(v_{s_1t_1}) = t(v_{s_0t_1})t(v_{s_1t_0})$. Writing out the corresponding polynomials,
$$
t(v_{s_0t_0}) = t(v_{s_1t_0}) = r_0+r_1+r_2,\quad t(v_{s_0t_1}) = t(v_{s_1t_1}) = r_3+r_4,
$$
we see that this is satisfied. The staged tree $(\T,\theta)$ additionally satisfies the property that for any two vertices $v,w$ in the same stage, $t(v) = t(w)$. In \cite[Lemma 2.13]{AD19} it was shown that this property implies balanced. To ``unbalance" the staged tree, one could re-color as blue one of the red vertices, changing the resulting floret. For another example of an unbalanced staged tree see \cite[Example 2.14]{AD19}.
\end{example}

One can determine statistical and algebraic properties of a staged tree model from its interpolating polynomial \cite{BGRS18, GS18} and  whether or not it is balanced \cite{AD19, DG20}. In particular the following proposition follows from \cite[Thm. 10]{DG20} and applies to our definition of staged tree models due to \cite[Lem. 5.13]{AD19}

\begin{prop}\label{prop:BalancediffToric}
A staged tree model $\Mcal_\T$ is equal to the toric model $\tMT$ if and only if $(\T,\theta)$ is balanced.
\end{prop}

\begin{example} For the staged tree $(\T,\theta)$ introduced in Example \ref{ex:firststagedtree}, we can define the toric model $\tMT$ by the monomial map $\phi_\T$.

\noindent
\begin{minipage}[t]{0.52\textwidth}
This is equivalent to the monomial map $\phi_{A_\T}$ arising from the \multipart matrix $A_{\T}$ on the right.

Thus $\tMT$ is the log-linear model $\Mcal_{A_{\T}}$. In Example \ref{ex:unbalanced} we showed that $(\T,\theta)$ is a balanced staged tree, and hence Proposition \ref{prop:BalancediffToric} implies that $\Mcal_\T = \tMT = \Mcal_{A_{\T}}$.
\end{minipage} \hfill
\begin{minipage}[t]{0.45\textwidth}
\vspace{-1cm}
\begin{center}
\centering\raisebox{\dimexpr \topskip-\height}{%
\scalebox{0.85}{
\begin{tikzpicture}
\draw[] (9.5,0) node{$A_{\T} = \begin{pmatrix}
 1 & 1 & 1 & 1 & 1 & \cdot & \cdot & \cdot & \cdot & \cdot \\
 \cdot & \cdot &  \cdot & \cdot & \cdot  & 1 & 1 & 1 & 1 & 1 \\
 \hdashline
 1 & 1 & 1  & \cdot & \cdot & 1 & 1 & 1  & \cdot & \cdot \\
 \cdot & \cdot & \cdot  & 1 & 1  & \cdot & \cdot & \cdot & 1 & 1 \\
 \hdashline
  1 & \cdot & \cdot & \cdot  & \cdot & 1 & \cdot & \cdot & \cdot & \cdot \\
 \cdot & 1 & \cdot & \cdot & \cdot & \cdot & 1 & \cdot & \cdot & \cdot  \\
 \cdot & \cdot & 1 & \cdot &  \cdot & \cdot & \cdot & 1 & \cdot & \cdot \\
 \cdot & \cdot & \cdot & 1 & \cdot &  \cdot & \cdot & \cdot & 1 & \cdot \\
 \cdot & \cdot & \cdot & \cdot & 1 & \cdot & \cdot & \cdot & \cdot  & 1  \\
 \end{pmatrix} $};
\draw[] (13.1,1.9) node{\footnotesize $s_{0}$};
\draw[] (13.1,1.4) node{\footnotesize $ s_{1} $};
\draw[] (13.1,1) node{\footnotesize $t_{0}$};
\draw[] (13.1,0.5) node{\footnotesize $t_{1}$};
\draw[] (13.1,0) node{\small $r_{0}$};
\draw[] (13.1,-0.4) node{\small $r_1$};
\draw[] (13.1,-0.9) node{\small $r_2$};
\draw[] (13.1,-1.4) node{\small $r_3$};
\draw[] (13.1,-1.9) node{\small $r_4$};
%\draw[] (7.5,-2.4) node{\tiny $p_{000}$};
%\draw[] (8.05,-2.4) node{\tiny $p_{001}$};
%\draw[] (8.6,-2.4) node{\tiny $p_{002}$};
%\draw[] (9.15,-2.4) node{\tiny $p_{010}$};
%\draw[] (9.7,-2.4) node{\tiny $p_{011}$};
%\draw[] (10.25,-2.4) node{\tiny $p_{100}$};
%\draw[] (10.8,-2.4) node{\tiny $p_{101}$};
%\draw[] (11.35,-2.4) node{\tiny $p_{102}$};
%\draw[] (11.9,-2.4) node{\tiny $p_{110}$};
%\draw[] (12.45,-2.4) node{\tiny $p_{111}$};
\end{tikzpicture}
} }
\end{center}
\end{minipage}
\end{example}

In the case that $(\T,\theta)$ is stratified and balanced, the staged tree model $\Mcal_\T$ is also a partition model. Since the tree is stratified the associated matrix $A_\T$ with the monomial map $\phi_\T$ (defined in \eqref{eq:TreeToricMap}) is a \multipart matrix, and being balanced implies that $\Mcal_\T=\widetilde{\Mcal}_{\T_A}=\Mcal_{A_\T}$. 

Note that the set of models defined by \multipart matrices $A$ whose associated tree $\T_A$ is a balanced, stratified staged tree is larger than the set of staged tree models arising from balanced, stratified staged trees. When the matrix $A$ has repeated columns, this information is lost when constructing the tree graph $\T_A$. Thus the staged tree model arising from $(\T_A,\theta_A)$ is the model obtained by removing repeated columns from a staged matrix $A$.

\begin{example}\label{ex:laststagedtree} Consider the matrix $A$ from Example \ref{ex:connectionratio}, which is used throughout Section \ref{sec:GRIP}. On the left of Figure \ref{Fig:TreeFlorets} the associated tree $(\T_A,\theta_A)$ is drawn and on the right we color portions of the matrix corresponding to the blue and red florets. We draw double circles around four of the leaves to indicate that the associated column is repeated in the matrix $A$. However, the underlying labeled tree graph is exactly the same as in Examples \ref{ex:firststagedtree} and \ref{ex:unbalanced}, and hence is a stratified and balanced staged tree. 

\scalebox{0.9}{
\begin{tikzpicture}[grow=right, edge from parent/.style={draw,-latex}, rounded corners]
\node[circle, draw, line width = 0.75pt]  {}
child[level distance = 8mm, sibling distance=3cm, line width = 0.75pt] {node[circle,draw, fill= violet!40] {} { 
	child[sibling distance=1.5cm, level distance = 9mm]{node[circle,draw,fill= red!50] {}
		child[sibling distance=0.5cm]{node[circle,draw, double, red!50] {}
		edge from parent[double] node [below, align=center] {\small $r_{4}$}}
		child[sibling distance=0.5cm]{node[circle,draw, double, red!50] {}
		edge from parent[double]}
	 edge from parent node [->, below, align=center]{\small $t_{1}$}          
	}
	child[sibling distance=1.5cm , level distance = 9mm]{node[circle,draw, fill=blue!30] {}
		child[sibling distance=0.6cm]{node[circle,draw, blue!50!white] {}}
		child[sibling distance=0.6cm]{node[circle,draw, blue!50!white] {}}
		child[sibling distance=0.6cm]{node[circle,draw, blue!50!white] {}}     
	}                                
}edge from parent node [left, align=center] {\small $s_{1}$}} 
child[level distance = 8mm,sibling distance=3cm, line width = 0.75pt] {node[circle,draw, fill= violet!40] {} 
	child[sibling distance=1.5cm, level distance = 9mm]{node[circle,draw, fill=red!50] {} 
		child[sibling distance=0.5cm]{node[circle,draw, double, red!50] {}
		edge from parent[double] node [below, align=center] {\small $r_{4}$}}
		child[sibling distance=0.5cm]{node[circle,draw, double, red!50] {}
		edge from parent[double] node [above, align=center] {\small $r_{3}$}}   
		 edge from parent node [->, above, align=center]
                {\small $t_{1}$}               
	}
	child[sibling distance=1.5cm, level distance = 9mm]{node[circle,draw, fill=blue!30]{}
		child[sibling distance=0.6cm,draw]{node[circle,draw, blue!50!white] {}
		edge from parent node [below, align=center] {\small $r_{2}$}}
		child[sibling distance=0.6cm]{node[circle,draw, blue!50!white] {}
		edge from parent node [above=-3.5pt, align=center, xshift=2pt] {\small $r_{1}$}}
		child[sibling distance=0.6cm]{node[circle,draw, blue!50!white] {}
		edge from parent node [above, align=center] {\small $r_{0}$}}
		 edge from parent node [->, above, align=center]
                { \small $t_{0}$}     
	}   edge from parent node [left, align=center] {\small $s_{0}$}                        
};
\draw[blue!20, fill =blue!20] (10,1.7) rectangle (11.5,-1.55);
\draw[blue!20, fill =blue!20] (6.15,1.7) rectangle (7.75,-1.55);
\draw[red!20, fill =red!20!white] (7.85,1.7) rectangle (9.9,-1.55);
\draw[red!20, fill =red!20!white] (11.6,1.7) rectangle (13.6,-1.55);
\draw[] (3.25,2.85) node{\footnotesize $p_{000}$};
\draw[] (3.25,2.25) node{\footnotesize $p_{001}$};
\draw[] (3.25,1.6) node{\footnotesize $p_{002}$};
\draw[] (3.65,1) node{\footnotesize $p_{010}, p_{010}'$};
\draw[] (3.65,0.5) node{\footnotesize $p_{011}, p_{011'}$};
\draw[] (3.25,-.15) node{\footnotesize $p_{100}$};
\draw[] (3.25,-0.75) node{\footnotesize $p_{101}$};
\draw[] (3.25,-1.35) node{\footnotesize $p_{102}$};
\draw[] (3.65,-2) node{\footnotesize $p_{110}, p_{110}'$};
\draw[] (3.65,-2.5) node{\footnotesize $p_{111}, p_{111}'$};
\draw[] (9.5,0.5) node{$A = \begin{pmatrix}
 1 & 1 & 1 & 1 & 1 & 1 & 1  & \cdot & \cdot & \cdot & \cdot & \cdot & \cdot & \cdot\\
 \cdot & \cdot & \cdot & \cdot & \cdot & \cdot & \cdot  & 1 & 1 & 1 & 1 & 1 & 1 & 1 \\
 \hdashline
 1 & 1 & 1  & \cdot & \cdot & \cdot & \cdot & 1 & 1 & 1  & \cdot & \cdot & \cdot & \cdot\\
 \cdot & \cdot & \cdot  & 1 & 1 & 1 & 1 & \cdot & \cdot & \cdot & 1 & 1 & 1 & 1\\
 \hdashline
  1 & \cdot & \cdot & \cdot & \cdot & \cdot & \cdot & 1 & \cdot & \cdot & \cdot & \cdot & \cdot & \cdot \\
 \cdot & 1 & \cdot & \cdot & \cdot & \cdot & \cdot & \cdot & 1 & \cdot & \cdot & \cdot & \cdot & \cdot \\
 \cdot & \cdot & 1 & \cdot & \cdot & \cdot & \cdot & \cdot & \cdot & 1 & \cdot & \cdot & \cdot & \cdot \\
 \cdot & \cdot & \cdot & 1 & \cdot & 1 & \cdot & \cdot & \cdot & \cdot & 1 & \cdot & 1 & \cdot \\
 \cdot & \cdot & \cdot & \cdot & 1 & \cdot & 1 & \cdot & \cdot & \cdot & \cdot & 1 & \cdot & 1 \\
 \end{pmatrix} $};
%\draw[] (14.075,2.5) node{$C^{\ell}_{i}$};
\draw[] (14.15,2.9) node{$E$};
%\draw[] (14.05,1.9) node{\footnotesize 7};
%\draw[] (14.3,1.4) node{\footnotesize 7};
%\draw[] (14.05,1) node{\footnotesize $\frac{3}{7}$};
%\draw[] (14.3,0.5) node{\footnotesize $\frac{4}{7}$};
%\draw[] (14.05,0) node{\small $\frac{1}{3}$};
%\draw[] (14.3,-0.4) node{\small $\frac{1}{3}$};
%\draw[] (14.05,-0.9) node{\small $\frac{1}{3}$};
%\draw[] (14.3,-1.4) node{\small $\frac{2}{4}$};
%\draw[] (14.05,-1.9) node{\small $\frac{2}{4}$};

\draw[] (14.05,2.4) node{\footnotesize $s_{0}$};
\draw[] (14.25,1.9) node{\footnotesize $ s_{1} $};
\draw[] (14.05,1.5) node{\footnotesize $t_{0}$};
\draw[] (14.25,1) node{\footnotesize $t_{1}$};
\draw[] (14.05,0.5) node{\small $r_{0}$};
\draw[] (14.25,0.1) node{\small $r_1$};
\draw[] (14.05,-0.4) node{\small $r_2$};
\draw[] (14.25,-0.9) node{\small $r_3$};
\draw[] (14.05,-1.4) node{\small $r_4$};
\draw[dashed, darkgray] (13.75,1.7)--(14.5,1.7);
\draw[dashed, darkgray] (13.75,0.73)--(14.5,0.73);
\draw[] (10,-2.25) node { \color{blue} $\mathcal{F}_{v_{s_0t_0}} = \mathcal{F}_{v_{s_1t_0}} = \{r_{0}, r_{1}, r_{2}\}$ \quad \color{red} $\mathcal{F}_{v_{s_0t_1}} = \mathcal{F}_{v_{s_1t_1}} = \{r_{3},r_{4}\}$};
\end{tikzpicture}
}
 \captionsetup{width=.8\linewidth}
 
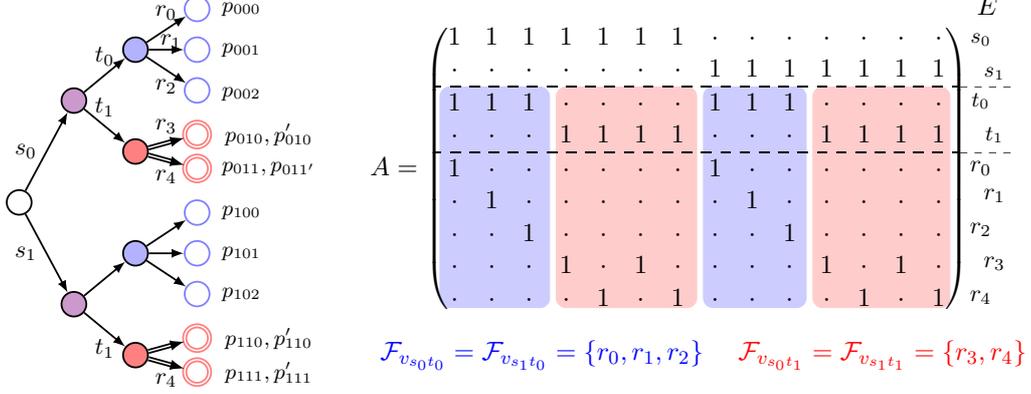
\captionof{figure}{The stratified and balanced staged tree $\mathcal{T}$ on the left corresponding to the matrix $A$ on the right.} \label{Fig:TreeFlorets}
\end{example}

We show that the stratified and balanced conditions on a staged tree are related to the generalized running intersection property for \multipart matrices.

\begin{lemma}\label{lem:treeWC}
The  \multipart matrix associated to a stratified and balanced staged tree is well-connected.
\end{lemma}

\begin{proof}
For a balanced and stratified staged tree we have that $\Mcal_A=\Mcal_{\T_A}$ and each column of $A$ represents a distinct root-to-leaf path $\lambda_j$ in $\T_A$. Consider an edge $(v,v')$ of $\T_A$ with label $s^\ell_i$, $1\leq \ell\leq m$. Then there exists a subset of $I\subset I^\ell_i \subset \{1,\hdots,m\}$ representing the set of root-to-leaf paths containing $(v,v')$, i.e.
$$
\{ \lambda\in \Lambda\, |\, (v,v') \in E(\lambda)\} = \{ \lambda_j\,|\, j\in I\subset I^\ell_i\}.
$$
The number of these paths, $|I|$ is exactly $c^\ell_j$ for any $j\in I$ since $c^\ell_j$ describes the number of repeated columns of the first $\ell$ partitions each of which is a distinct column of $A$. Thus the number of terms in the summation
$$
t(v') = \sum_{\lambda \in \Lambda_v} \prod_{s\in \theta(\lambda)} s = \sum_{\lambda_j\, |\, j\in I} \prod_{s\in \theta(\lambda_j)}s
$$
is $c^\ell_j$. In other words if we denote $\sigma(t(v'))$ as the sum of the coefficients of $t(v')$, then $\sigma(t(v'))=c^\ell_j$. Note that, for any $v,w\in V$, since $t(v)$ is a polynomial with positive integers coefficients, $\sigma(t(v)t(w)) = \sigma(t(v))\sigma(t(w))$.

Now fix a row $\alpha^\ell_i$ in the $\ell$-th partition of $A$ and an index $j\in I^\ell_i$. Let $(v,v')$ be the edge in $\T_a$ such that $\theta(v,v')=s^\ell_i$ and $(v,v')\in \theta(\lambda_j)$. Similarly let $\tilde{j}\in I^\ell_i$ be any other index in $I^\ell_i$ and $(w,w')$ be the edge such that $\theta(w,w')=s^\ell_i$ and $(w,w')\in \theta(\lambda_{\tilde{j}})$. Since $\T_A$ is balanced, for any edges $(v,v'')\in E(v)$ and $(w,w'')\in E(w)$ we have
$t(v')t(w'')=t(w')t(v'').$
In particular we can sum over all the vertices in $\ch(v)$ and $\ch(w)$:
\[
t(v')\sum_{w''\in \ch(w)} t(w'') = t(w') \sum_{v''\in \ch(v)} t(v''), \]
and hence $t(v')t(w)=t(w')t(v)$.
Similarly as before, the number of paths containing an edge $(v,v'')$ or $(w,w'')$ for $v''\in \ch(v), w''\in \ch(w)$ is $c^{\ell-1}_j$ or $c^{\ell-1}_{\tilde{j}}$ respectively. Thus $\sigma(t(v)) = c^{\ell-1}_j$ and $\sigma(t(w))= c^{\ell-1}_{\tilde{j}}$. Then the above equality implies that
$\sigma(t(v')t(w))=\sigma(t(w')t(v))$. Thus $c^{\ell}_jc^{\ell-1}_{\tilde{j}}=c^{\ell}_{\tilde{j}}c^{\ell-1}_j$,
which proves the result.
\end{proof}

\begin{theorem}\label{thm:balanced2GRIP}
The associated \multipart matrix for a stratified and balanced staged tree satisfies the GRIP.
\end{theorem}

\begin{proof}
By Lemma \ref{lem:treeWC} we know that the associated \multipart matrix $A$ is well-connected, and satisfies the floret condition by assumption. Thus it remains to show that for each $1\leq \ell < k$ we have that the rows of $B_{\ell} \Cap A^{\ell+1}$ lie in the rowspan of $A^{1,\dots,\ell}$ where $B^\ell = \Cup_{m=1}^\ell A^\ell$.

In the proof of \cite[Theorem~2.5]{AD19}, the authors compute the generators of the toric ideal $I(A^{1,\dots,\ell})$. In order to accomplish this, they show in \cite[Proposition~4.5]{AD19} that $I(A^{1,\dots,\ell})$ is multihomogeneous with respect to the grading given by $D$. In particular, this implies that for all $b$ in the integer kernel of $A^{1,\dots,\ell}$ and all rows $\alpha^D$ of $D$, $\alpha^D \cdot b = 0$. Since $A^{1,\dots,\ell}$ is an integer matrix, there exists an integer basis for its kernel. Thus for each row $\alpha^D$ of $D$, we have $\alpha^D \in (\ker(A^{1,\dots,\ell}))^{\perp} = \mathrm{rowspan}(A^{1,\dots,\ell})$, as needed.
\end{proof}

Thus stratified and balanced staged tree models lie in the family of partition models whose matrix satisfies the GRIP, which by Theorem \ref{thm:GRIPOneCycle} implies the following corollary.

\begin{cor}\label{cor:BalancedOneCycle}
For the \multipart matrix associated with a stratified and balanced staged tree, the IPS algorithm results in the MLE after one cycle.
\end{cor}

%\begin{proof}
%This follows directly from Theorems \ref{thm:GRIPOneCycle} and \ref{thm:balanced2GRIP}.
%\end{proof}

\begin{remark}\label{rem:stagedtreeMLE}
Consider the model $\Mcal_\T$ obtained from a stratified staged tree $\T$ with $k$ levels. Denote the florets of $\T$ by $\Fcal^\ell_{i}$ where $1\leq \ell \leq k$ indexes the levels of $\T$ and $1\leq i\leq n_{\ell}$ indexes the edge labels of the $\ell$th level of $\T$; thus $\Fcal^\ell_i$ is the floret associated with edge label $s^\ell_i$. As discussed in Remark \ref{rem:MatrixFloretFull} there is a one to one correspondence between the florets of $\T$ and of $A_\T$. Let $A_\T$ be the associated \multipart matrix with $\T$, and $d$ be a data vector. Then by \cite[Proposition 11]{DMS19} the MLE of $\Mcal_\T$ is given by
\begin{equation}\label{eq:StagedTreeMLE}
\p(d)=\phi_{A_\T}\left(s^1_1(d),\hdots,s_{n_1}^1(d),s_1^2(d),\hdots,s^k_{n_k}(d)\right)
\end{equation}
where $s^\ell_i(d) = \frac{\alpha^\ell_i(d)}{\displaystyle\sum_{\alpha_{i'}^\ell\in \Fcal_i^\ell}a_{i'}^\ell(d)}$.
\end{remark}

We now show that the GRIP is, in some sense, a characterization of stratified and balanced staged trees. All \multipart matrices that satisfy the GRIP, in fact, have an associated tree that is a stratified and balanced staged tree.

\begin{theorem}\label{thm:GRIP2balanced}
For any \multipart matrix $A$ satisfying the GRIP, the associated tree $\T_A$ is a stratified and balanced staged tree.
\end{theorem}
\begin{proof}
Consider the matrix $\bar{A}$ obtained by removing the repeated columns of $A$. The associated tree $\T_{\bar{A}}$ is equivalent to $\T_A$, and hence is a stratified staged tree. Since there are no repeated columns, we have that the toric model for the staged tree $\T_A$ is given by the partition model of $\bar{A}$, i.e. $\widetilde{\Mcal}_{\T_A} = \Mcal_{\bar{A}}$.

By Proposition \ref{prop:BalancediffToric}, the staged tree model $\Mcal_{\T_A}$ is equal to $\Mcal_{\bar{A}}$ if and only if $\T_A$ is a balanced staged tree. By Proposition \ref{prop:RemoveRepeatMLE} the MLE of $\Mcal_{\bar{A}}$ is equal to the MLE of the staged tree model $\T_A$ (Remark \ref{rem:stagedtreeMLE}). Since the maximum likelihood estimator also parameterizes the model, this implies that
$\widetilde{\Mcal}_{\T_A} = \Mcal_{\bar{A}} = \Mcal_{\T_A}.$
Hence by Proposition \ref{prop:BalancediffToric}, $\T_A$ is balanced.
\end{proof}

Putting the previous results together, we can say that a \multipart matrix without repeated columns satisfies the GRIP if and only if the associated tree is a stratified and balanced staged tree. We also note that this implies that the IPS algorithm provides another method for determining whether a stratified staged tree is balanced. If the IPS algorithm produces the MLE exactly in one cycle for an arbitrary data vector and, in each case, the MLE is given by \eqref{eq:StagedTreeMLE}, then the stratified staged tree is balanced. Otherwise the stratified staged tree is not balanced via Corollary \ref{cor:BalancedOneCycle}.

\section{Discussion}\label{sec:disc}

In this work we explored the relationship between \multipart matrices, their associated partition model, and the IPS algorithm. In particular we showed equivalences between the GRIP and other conditions from disparate areas of algebraic statistics and examined their implications for exact convergence under the IPS algorithm. We conclude with a discussions of some natural questions that arise from these results.

\begin{question}
For a \multipart matrix $A$ with $k$ partitions, does the requirement that the IPS algorithm produces the MLE for $A^{1,\hdots,\ell}$ at each step imply that $A$ must satisfy the GRIP?
\end{question}

In Remark \ref{rem:Horn}, we claim that IPS producing the MLE in 2 steps for a \multipart matrix with 2 partitions is enough to guarantee that the GRIP is satisfied. While the GRIP is sufficient for the algorithm to produce the MLE for $A^{1,\hdots,\ell}$ at each step, it is possible that the reverse logical direction holds.

\begin{question}
Let $\Mcal$ be a log-linear partition model with rational MLE. Does there always exist a \multipart matrix $A$ representing $\Mcal$ that satisfies the GRIP?
\end{question}

It is clear that the matrix representation of a particular model affects the IPS algorithm. In \cite{CS20}, the authors show sufficient conditions for 2-way quasi-independence models to have rational MLE. Since $k$-way quasi-independence models are just partition models without repeated columns, it would be interesting to see if these results can be used to show that every such matrix  has a representation satisfying the GRIP. Indeed our preliminary computations indicate that this is true.

Finally there are many results and questions related to the convergence of the generalized IPS algorithm on log-linear models and its connections to tools from algebraic statistics (see \cite[Sec. 5]{AKRS21} and \cite[Sec. 7.3]{drton2008}). Our work falls adjacent to this line of inquiry and that it would be worthwhile to investigate whether tools that have recently produced results in this area connect to our work on the generalized running intersection property.

\section*{Acknowledgments}
We would like to thank Bernd Sturmfels, Nihat Ay, Ji\u{r}\'i Vomlel, Eliana Duarte, Orlando Marigliano and Carlos Am\'endola for helpful discussions and suggestions regarding this project.  Jane Coons was partially supported by the US National Science Foundation (DGE
1746939) and Carlotta Langer was supported by the priority programme “The Active Self” (SPP 2134) of the Deutsche Forschungsgemeinschaft. This work was also supported in part  by the Max Planck Institute for Mathematics in the Sciences and the Data Institute, University of San Francisco.

\section*{Declaration of interests}
The authors declare that they have no known competing financial interests or personal relationships that could have appeared to influence the work reported in this paper.

\bibliographystyle{acm}
\small
\bibliography{arXivVersion-IPS}

\end{document}